\documentclass[12pt]{amsart} 
\usepackage[left=3cm,top=2.5cm,bottom=2.5cm,right=3cm]{geometry}
\usepackage{times}
\usepackage{amssymb, amsmath, amsthm}
\usepackage{mathtools}
\usepackage{graphicx,xspace}
\usepackage{epsfig}
\usepackage{enumitem}
\usepackage[usenames,dvipsnames]{xcolor}
\usepackage{tikz}
\usepackage[T1]{fontenc}
\usepackage[utf8]{inputenc}
\usepackage{bm}
\usepackage[config, labelfont={normalsize}]{caption,subfig}
\usepackage{mathrsfs}
\definecolor{green}{RGB}{0,127,0}
\definecolor{red}{RGB}{191,0,0}
\usepackage[colorlinks,cite color=red,link color=green,pagebackref=true]{hyperref}
\usepackage{todonotes}

\usetikzlibrary{matrix,arrows,calc}
\usetikzlibrary{decorations.markings}
\usetikzlibrary{patterns}
\usetikzlibrary{snakes}
\usetikzlibrary{shapes}

\makeatletter
\renewcommand{\tocsection}[3]{%
\indentlabel{\@ifnotempty{#2}{\bfseries\ignorespaces#1 #2.\quad}}\bfseries#3}
\renewcommand{\tocsubsection}[3]{%
\hspace{1cm} \indentlabel{\@ifnotempty{#2}{\ignorespaces#1 #2.\quad}}#3}
 \makeatother

\usepackage[capitalize]{cleveref}

\theoremstyle{plain}

\newtheorem{lemma}{Lemma}[section]
\newtheorem{theorem}[lemma]{Theorem}
\newtheorem{corollary}[lemma]{Corollary}
\newtheorem{proposition}[lemma]{Proposition}

\newtheorem{defprop}[lemma]{Definition-Proposition}
\theoremstyle{definition}
\newtheorem{definition}[lemma]{Definition}
\newtheorem{definition-lemma}[lemma]{Definition-Lemma}

\theoremstyle{remark}
\newtheorem{remark}{Remark}

\newcommand{\xdownarrow}[1]{{\left\downarrow\vbox to #1{}\right.\kern-\nulldelimiterspace}}
\newcommand{\splus}{\!+\!}
\newcommand{\sminus}{\!-\!}
\newcommand{\spm}{\!\pm}

\newcommand{\RR}{\mathbb{R}}
\newcommand{\C}{\mathbb{C}}
\newcommand{\N}{\mathbb{N}}

\newcommand{\Sym}[1]{\mathfrak{S}_{#1}}

\newcommand{\Z}{\mathbb{Z}}

\newcommand{\QQ}{\mathbb{Q}}

\newcommand{\pp}{\mathbf{p}}
\newcommand{\qq}{\mathbf{q}}

\DeclareMathOperator{\Pf}{Pf}
\def\la{\lambda}

\DeclareMathOperator{\sgn}{sgn}

\DeclareMathOperator{\Gl}{GL}

\newcommand{\m}{\mathfrak{h}}

\DeclareMathOperator{\Symm}{Sym}

\DeclareMathOperator{\hook}{hook}
\DeclareMathOperator{\Res}{Res}

\title[$b$-monotone Hurwitz numbers]{$b$-monotone Hurwitz numbers:\\ Virasoro constraints, BKP hierarchy, and $O(N)$-BGW integral}

\author[V.~Bonzom]{Valentin Bonzom}
\address{Universit\'e Sorbonne Paris Nord, LIPN, CNRS, UMR 7030, F-93430 Villetaneuse, France}
\email{bonzom@lipn.univ-paris13.fr}

\author[G.~Chapuy]{Guillaume Chapuy}
\address{CNRS, IRIF UMR 8243, Universit\'e de Paris.}
\email{guillaume.chapuy@irif.fr}

\author[M.~Dołęga]{Maciej Dołęga}
\address{
Institute of Mathematics, 
Polish Academy of Sciences, 
ul. Śniadeckich 8, 
00-956 Warszawa, Poland.
}
\email{mdolega@impan.pl}

\thanks{This project has received funding from the European Research
  Council (ERC) under the European Union’s Horizon 2020 research and
  innovation programme (grant agreement No. ERC-2016-STG 716083
  “CombiTop”). MD is supported by {\it Narodowe Centrum Nauki},
  grant UMO-2017/26/D/ST1/00186. VB is partially supported by the ANR-20-CE48-0018 3DMaps.}

\begin{document}

\begin{abstract}
We study a $b$-deformation of monotone Hurwitz numbers, obtained by
deforming Schur functions into Jack symmetric functions. We give an
evolution equation for this model and derive from it Virasoro
constraints, thereby proving a conjecture of F\'eray on Jack
characters. A combinatorial model of non-oriented monotone Hurwitz
maps which generalizes monotone transposition factorizations is
provided.

In the case $b=1$ we obtain an explicit Schur expansion of the model and show that it obeys the BKP integrable hierarchy. This Schur expansion also proves a conjecture of Oliveira--Novaes relating zonal polynomials with irreducible representations of $O(N)$. We also relate the model to an $O(N)$ version of the Br\'ezin--Gross--Witten integral, which we solve explicitly in terms of Pfaffians in the case of even multiplicities.



      \end{abstract}

      \maketitle

\section{Introduction}

      In recent years, the enumerative properties of branched
      coverings of the sphere by orientable surfaces, that is, the
      study  of \emph{Hurwitz numbers}, has received a considerable
      attention~\cite{GouldenJackson1997a,EkedahlLandoShapiroVainshtein2001,GraberVakil2003,GouldenJacksonVakil2005,OkounkovPandharipande2006,
        KazarianLando2007, Guay-PaquetHarnad2017}. The subject widely intersects the study of \emph{map enumeration}, which has been active in combinatorics since the 1960s, and the study of \emph{matrix models}, whose topological expansions give rise to enumerative series of maps or Hurwitz numbers of various kinds (e.g.~\cite{LandoZvonkin2004,Eynard:book}).
A rich class of models in the area is given by the \emph{weighted Hurwitz numbers} of~\cite{Guay-PaquetHarnad2017}. They are parametrized by a certain weight function, hereafter denoted $G(z)$. Several classical models of enumerative geometry, such as Hurwitz numbers and dessins d'enfants, can be recovered for some choices of $G(z)$.

One of the most studied models among them are \emph{monotone Hurwitz numbers}, corresponding to the weight function $G(z)=\tfrac{1}{1-z}$. They have a nice combinatorial interpretation in terms of constrained transposition walks in the symmetric group~\cite{GouldenGuayPaquetNovak2014}. They have been proved in~\cite{Novak2020} to give an explicit combinatorial meaning to the $1/N$-expansion
      the Harish-Chandra/Itzykson--Zuber (HCIZ) and Brézin--Gross--Witten
      (BGW) integrals over $U(N)$. 
      Additionally, it was shown that they posses rich structural
      similarities with classical Hurwitz numbers: they are closely related to
      hypergeometric tau functions \cite{HarnadOrlov2015}, they satisfy the
       topological recursion of Eynard-Orantin \cite{DoDyerMatthews2017}, and the structural similarity between low-genus formulas is remarkable~\cite{GouldenGuayPaquetNovak:polynomiality}.

      \smallskip 

      In this paper, we will be interested in deformations of these models by a one-dimensional parameter called $b$, or $\beta=2/(1\!+\!b)$. These deformations have two different origins.

	  The first one lies in the theory of matrix models, where the
      \emph{$\beta$-ensembles}~\cite{Mehta2004} interpolate between
      the several types of standard ensembles. In the 1-matrix model, the $\beta$-ensemble interpolates between the Gaussian Unitary Ensemble at $\beta=2$, which consists in random Hermitian matrices, and the Gaussian Orthogonal Ensemble at $\beta=1$, which consists in random real, symmetric matrices. For these two
      special values of $\beta$, these models famously give rise to topological
      expansions which correspond respectively to the enumeration of
      maps on orientable and on non-oriented surfaces\footnote{We will
        use the word ``non-oriented'' to denote a general surface (a
        real, compact, 2-dimensional manifold without boundary) which may be orientable or not.}. For arbitrary values of $\beta$, the topological expansion over non-oriented surfaces is still valid but the latter are counted with a certain weight which is a monomial in the parameter\footnote{We will follow the habits of the combinatorial literature and favour the parameter $b$ over the equivalent parameter $\beta=2/(1\!+\!b)$, since it is the one which gives rise to nice positive expansions.} $b=(2-\beta)/\beta$, as shown by La~Croix~\cite{LaCroix2009}.
Although $\beta$-deformations of monotone Hurwitz numbers have not been considered prior to~\cite{ChapuyDolega2020} and this work, an arbitrary-$\beta$ 2-matrix model was studied in \cite{BergereEynardMarchalPrats-Ferrer2012}, based on a $\beta$-deformation of the HCIZ integral proposed in \cite{BergereEynard2009}. As for the BGW model, a $\beta$-deformation was proposed in \cite{MironovMorozovShakirov2010} in terms of an expansion on Jack symmetric functions, and it coincides with the function we study here (indeed the paper \cite{MironovMorozovShakirov2010} already derived the Virasoro constraints, via a technique of pure gauge limit; see Remark~\ref{rem:MMS11} below).

      The second origin lies in algebraic
      combinatorics. In~\cite{GouldenJackson1996}, Goulden and Jackson
      have introduced a $b$-deformation of the generating function of
      bipartite maps, obtained by considering the expansion of this
      function in  Schur functions, and replacing them by Jack
      symmetric functions of parameter $\alpha=1+b$ (see below). This
      model interpolates between the cases of orientable ($b=0$) and
      non-oriented ($b=1$) surfaces, and their \emph{$b$-conjecture},
      a fascinating open problem in algebraic combinatorics, claims
      that this function still has an explicit topological expansion
      for arbitrary values of the parameter $b$. The $\beta$-ensemble
      can be viewed as a special case of the Goulden-Jackson function,
      thus the previously mentioned topological expansion of La~Croix solves a special case of the $b$-conjecture, see~\cite{LaCroix2009}.
     
\smallskip

In the recent paper~\cite{ChapuyDolega2020}, the last two authors have extended this $b$-deformation to the whole family of tau functions of weighted Hurwitz numbers, i.e. for an arbitrary weight function $G(z)$. The corresponding coefficients are shown to enumerate generalized branched coverings of the sphere by non-oriented surfaces, or equivalently some maps on these surfaces called constellations, with an appropriate combinatorial $b$-weighting scheme. More precisely, the function of~\cite[Section 6]{ChapuyDolega2020} is given by
\begin{equation}\label{eq:deftaub}
	\tau^G_b\equiv \tau_b^G (t, u; \mathbf{p},\mathbf{q}; b) := \sum_{n\geq 0} t^n \sum_{\lambda \vdash n} \frac{J^{(b)}_\lambda(\mathbf{p}) J^{(b)}_\lambda(\mathbf{q})}{\|J^{(b)}_\lambda\|^2} \prod_{\Box \in \lambda } G(u \cdot c_{b}(\Box)),
\end{equation}
where the second sum is taken over partitions $\lambda$ of the integer
$n$, where $J^{(b)}_\lambda$ denotes the Jack symmetric function of parameter $(1+b)$ indexed by the partition $\lambda$, expressed as a function of its power-sum variables ${\bf p}=(p_1,p_2,\dots)$ or ${\bf q}=(q_1,q_2,\dots)$, and where the product is taken over all boxes $\Box$ of the partition $\lambda$, whose $b$-content is denoted by $c_{b}(\Box)$. All these notions will be carefully defined in Section~\ref{subsec:Jack}. Here the weight function $G(z)=1+g_1 z + \dots $ is a formal power series whose coefficients can be considered as infinitely many additional free variables of the model. 

The main result of~\cite{ChapuyDolega2020} is that the log-derivative
of this function has an expansion with positive integer coefficients,
which can be interpreted as counting generalized branched coverings of
the sphere by non-oriented surfaces with certain weights, whose Euler
characteristic is tracked by the power of $u$ (see
also~\cite{BenDali2021}). This statement is strongly related to the
$b$-conjecture of Goulden and Jackson. In particular it implies a
special case of the $b$-conjecture and it generalizes La~Croix's results, see~\cite[Section~6.4]{ChapuyDolega2020} for a discussion.
For $b=0$, the function $\tau^G_b$ in  \eqref{eq:deftaub} is the 2-Toda tau function of (orientable) weighted Hurwitz numbers~\cite{OrlovScherbin2000,Guay-PaquetHarnad2017}.

\medskip

This naturally calls for studying these $b$-deformed weighted Hurwitz numbers in more depth, especially for choices of weight functions  $G(z)$ which are of special interest due to their connections with other objects in mathematics.
Part of this program has actually been completed long before these notions
were introduced, in the case of the $\beta$-ensemble for the 1-matrix model 
already mentioned. It corresponds to a linear $G(z)$ with the
specialization $q_i=\delta_{i,1}$. This $\beta$-ensemble enjoys a
collection of nice properties: explicit formulation as matrix integrals at $b=0$ and $1$ \cite{Mehta2004}, explicit Virasoro constraints \cite{AdlervanMoerbeke2001}, explicit combinatorial expansion \cite{ChapuyDolega2020}, tau function of the KP hierarchy at $b=0$ \cite{Kharchev91}
and of the BKP
hierarchy\footnote{In the literature, two different hierarchies bear
  the name ``BKP''. In this paper we use the ``large/charged'' BKP hierarchy
  of Kac and Van de Leur~\cite{KacVandeLeur1998}, whose tau functions have expansions on Schur functions with
  Pfaffian coefficients and come from the orbit of the group $O(2\infty+1) \supset GL(\infty)$ on the Dirac vacuum;
  not the ``small/neutral'' one of the Kyoto
  school~\cite{DateJimboKashiwaraMiwa1981/2} which is related to expansions in
  $Q$-Schur functions and to the (smaller) group $O(\infty) \subset
  GL(\infty)$ (and used for example in~\cite{Nimmo1990,Orlov2003}). As it turns out, our functions are \emph{not} solution of the small BKP hierarchy (nor in fact of the KP hierarchy).} of Kac and Van de Leur at $b=1$ \cite{VandeLeur2001}. It is important to stress that these properties are not expected to be shared by weighted Hurwitz numbers in general. On the contrary they reflect the importance of this particular choice of weight function. See Appendix~\ref{sec:otherModels} for some comments on the $\beta$-ensemble for the 1-matrix model, and for the closely related model of dessins d'enfants.
\medskip
With those motivations in mind and looking for a model which replicates the nice properties of the $\beta$-ensemble, we will consider in this paper the
``monotone Hurwitz'' case of the function $\tau_b^G$, that is to say
with the choice of weight function $G(z)=\tfrac{1}{1-z}$ (or a
rescaling of it, which we denote $\tau^Z_b$, see Definition~\ref{def:mainfunction} or
Remark~\ref{rem:scalings}). In the usual language of enumerative geometry, we will only consider
\emph{single} (rather than double) Hurwitz numbers, meaning that the
variable $q_i$ in \eqref{eq:deftaub} will be taken to
$\mathbf{q}=\mathbf{1}:=(\delta_{i,1})$. For $b=0$, this coincides
with the aforementioned notion of single monotone Hurwitz numbers,
therefore our object of studies is a non-oriented, $b$-deformed, extension of them.

We will show that this model essentially enjoys the same
nice properties listed above for the $\beta$-ensemble, namely relation to matrix integrals, Virasoro constraints,
KP and BKP hierarchies, and combinatorial topological expansion. However,
as is already the case for $b=0$, the case of monotone Hurwitz numbers
is arguably more delicate (see e.g. the discussion in Section~\ref{sec:formalNSolutions}). In particular, the techniques developed
in this paper are not the same as those
of~\cite{MironovMorozovSemenoff1996, ZinnJustin2002, Alexandrov2018}
for the BGW model ($b=0$),~\cite{VandeLeur2001} for maps at
$b=1$,~\cite{AdlervanMoerbeke2001} for $\beta$-ensembles,
and~\cite{MironovMorozovShakirov2010} for the ``pure gauge'' limit of
Selberg integrals. Those references indeed all rely on matrix integrals while here we instead use algebraic combinatorics and in fact we are eventually able to derive applications to matrix integrals.

Several notions appearing in our analysis give rise to
results of independent interest. We make a connection to characters of the orthogonal group for which we find some seemingly new Pfaffian expression. We further prove two conjectures coming from different contexts \cite{Feray2012,OliveiraNovaes2021} dealing with zonal and Jack polynomials, and orthogonal characters. We also apply our results to
evaluate the $O(N)$ BGW integral in terms of Pfaffians.

\medskip
\noindent We now present the organization of the paper, highlighting the main results.
\begin{itemize}[itemsep=0pt, topsep=0pt,parsep=0pt, leftmargin=12pt]
	\item In Section~\ref{sec:virasoro} we prove an evolution
          equation, or quantum spectral curve, which is a
          $\frac{d}{dt}$ equation for our main function $\tau^Z_b$
          (Theorem~\ref{thm:evolution}). 
		We use it to prove a set of
          Virasoro constraints (Theorem~\ref{thm:virasoro}). The
          evolution equation is 
		a direct consequence of the Virasoro constraints,
		but in our case it is important
          to proceed in the other direction, i.e. to deduce
          the constraints from this equation thanks to a lemma of
          independent interest (Lemma~\ref{lemma:evolutionImpliesVirasoro}).
		  We then use the
          Virasoro constraints to prove a conjecture of
		F\'eray~\cite{Feray2012} related to Jack characters (Theorem~\ref{thm:Feray}).

	\item In Section~\ref{sec:model} we introduce a simple
          combinatorial model of embedded graphs, whose exponential generating
          function is precisely $\tau^Z_b$. It directly extends the
          classical combinatorial model of non-deformed monotone
          Hurwitz numbers given by non-decreasing factorizations of a
          permutation into
          transpositions~\cite{GouldenGuayPaquetNovak2013}.

	\item In Section~\ref{sec:schurExpansion} we study the
          function $\tau^Z_b$ for $b=1$. We show that this function,
          which is defined by its expansion in zonal polynomials in
          this case, has an explicit expansion in Schur functions with rescaled variables, involving dimensions of irreducible orthogonal representations (Theorem~\ref{thm:schurExpansion}). This fact is strongly related to a conjecture of Oliveira and Novaes~\cite{OliveiraNovaes2021}, which we prove as a byproduct (Theorem~\ref{thm:OliveiraNovaes}).

	\item In Section~\ref{sec:BKP} we show that the coefficients in this scaled Schur expansion can be expressed as Pfaffians. We deduce  that a rescaling of the function $\tau^Z_{b=1}$ is a tau function of the large BKP hierarchy (Theorem~\ref{thm:mainBKP}). This statement involves a discussion about truncations of expansions and about the role of the charge parameter $N=u^{-1}$ which \emph{has to be} interpreted as a formal parameter rather than an integer (as in classical matrix models). %

	\item In Section~\ref{sec:BGW} we show that an orthogonal version of
          the BGW
          integral~\cite{BrezinGross1980,GrossWitten1980} can be expressed by specializing $\tau^Z_{b=1}$. We use
          results of the previous sections to give an explicit solution to
          this integral in the special case where the eigenvalues of the external matrix have even degeneracy, 
		  in terms of Pfaffians of modified Bessel functions
          (Theorem~\ref{thm:pfaffianBGW}). This solves the
          problem of finding a closed, explicit formula for the orthogonal
          BGW/HCIZ integral in the case when external matrices are
          diagonal ~\cite{BergereEynard2009}.

\end{itemize}
Finally in Appendix~\ref{sec:otherModels} we collect some comments and analogies with the two other models of arbitrary-$\beta$ 1-matrix model and dessins d'enfants (equivalently non-oriented general maps and non-oriented bipartite maps). In particular we use the techniques we have developed in the previous sections to make explicit some results on maps and bipartite maps which can be considered as known, even if they are not all explicitly written in the literature.

\section{Evolution equation and Virasoro constraints}
\label{sec:virasoro}

\subsection{Definition of the main function}
\label{subsec:Jack}
We start by reviewing some properties of Jack symmetric functions. This material is standard, see~\cite{Stanley1989,Macdonald1995} for complements and for notions which we do not define here.
Everywhere in the paper, $b$ is a formal or complex variable, and $\pp = (p_i)_{i \geq 1}$ denotes an infinite family of  variables. We denote by $p_i^*$ the differential operator
\begin{align}\label{eq:pistar}
	p_i^*:=\frac{i\partial}{\partial p_i}.
\end{align}

      \begin{definition}
The Laplace--Beltrami\footnote{the name comes from the fact that for $b=1$ this operator is related to the classical Laplace--Beltrami operator as shown in~\cite{James1968}; it is common in the literature to use the same name for the general $b$-case;
  see~\cite{Stanley1989} and the comment there following eq.~(11)} operator $D_b$ is the differential operator defined by
\begin{equation}
\label{eq:Laplace-Beltrami}
	D_{b}: = \frac{1}{2}\left((1+b)\sum_{i,j\geq 1}p_{i+j}p_i^* p_{j}^* +\sum_{i,j\geq
  1}p_{i}p_j p_{i+j}^*+b\cdot\sum_{i\geq
  1}(i-1)p_{i}p_i^*\right).
\end{equation}
\end{definition}

Let $\Symm_b=\QQ(b)[p_1,p_2,\dots]$ denote the polynomial algebra over the field of fractions $\QQ(b)$. It is
isomorphic to the algebra of symmetric functions over $\QQ(b)$ by
identifying the variables $\pp$ with
the power-sum basis, see~\cite{Stanley:EC2}. In particular the Laplace--Beltrami operator $D_b$ given by
\eqref{eq:Laplace-Beltrami} acts on the symmetric function
algebra. Another (linear) basis of $\Symm_b$ is given by the \emph{monomial symmetric functions} which we denote by $m_\lambda$. 

We can identify an integer partition $\lambda$ of $n$ (denoted
$\lambda \vdash n$ or $|\lambda|=n$) with its \emph{Young diagram}, which is the union of \emph{boxes} $(x,y)$ with $x\leq \ell(\lambda)$ and $y\leq \lambda_x$, where $\ell(\lambda)$ is the  number of parts of $\lambda$.
If $\square = (x,y)$ is a box of $\lambda$ we write $\square\in\lambda$ and we denote by $c_b(\square)$ its \emph{$b$-content}%
\footnote{This quantity is usually called the $\alpha$-content, and denoted by $c_\alpha(\square)$, with $\alpha=1+b$. Because the parameter $b$ is more natural than $\alpha$ in our context, we prefer to use this convention. %
}:
      \[ c_b(\square) = b(x-1)+x-y.\]

      \begin{defprop}
  \label{defprop:Jack}
There is a unique family of symmetric functions
$\{J^{(1+b)} _\lambda(\pp)\}$ such that for each partition $\lambda$,
\begin{itemize}
\item $D_b J^{(1+b)}_\lambda = \left(\sum_{\square \in \la}c_b(\square)\right)J^{(1+b)}_\lambda$;
\item
	$ J_\la^{(1+b)} = \hook_b(\la) m_\la + \sum_{\nu <
          \la}a^{\la}_\nu m_\nu,  \text{ where } a^{\la}_\nu \in
        \QQ(b)$, $<$ is the dominance order on integer partitions, and where
	\begin{align}
\label{eq:HookProduct}
\hook_b(\la) &:= \prod_{\square \in \la}\Big((1+b)\ a(\square) + \ell(\square) +1 \Big).
	\end{align}
  \end{itemize}
  We call them \emph{Jack symmetric functions} of parameter $(1+b)$.
\end{defprop}
We can endow $\Symm_b$ with a scalar product by defining it on
the basis of power-sum symmetric functions
\begin{equation}
  \label{eq:product}
  \langle p_\mu,p_\nu\rangle_b := (1+b)^{\ell(\mu)}z_\mu\delta_{\mu,\nu},
\end{equation}
where $\delta_{\mu,\nu}$ is the Kronecker delta, $z_\mu:= \prod_{i
  \geq 1}i^{m_i(\mu)}m_i(\mu)!$, and $m_i(\mu)$ is the number of parts
of $\mu$ equal to $i$. A fundamental fact is that Jack
symmetric functions are orthogonal, with the following squared norm:
\[\langle J_\la^{(1+b)},J_\la^{(1+b)}\rangle_b = \hook_b(\la)
  \hook'_b(\la) =: j_\la^{(b)},\]
where
\begin{align}
\label{eq:HookProduct2}
\hook'_b(\la) &:= \prod_{\square \in \la}\Big((1+b)\ a(\square) +
\ell(\square) +(1+b) \Big).
\end{align}
To prevent confusion, we emphasize that the operator $p_i^*$ in~\eqref{eq:pistar} is \emph{not} the dual of the multiplication by $p_i$ for $\langle \cdot, \cdot \rangle_b$ (it is for $\langle \cdot, \cdot \rangle_0$, although we will not directly use this fact).

In this paper we will use the notation $[\cdot], (\cdot), [[\cdot]]$ to denote respectively polynomials, rational functions, and formal power series; for example $\mathbb{Q}(b)[\pp][[u,t]]$ is the ring of formal power series in $t$ and $u$ whose coefficients are polynomials in the variables $p_i$ with coefficients themselves rational fractions in $b$ over $\mathbb{Q}$.
We now define our main function, with a slight abuse of notation with respect to~\eqref{eq:deftaub} since we will not need the variables $\qq$ anymore.
\begin{definition}[Main function]\label{def:mainfunction}
	We let $\tau_b^{Z}(t;\pp,u)$ be the generating function of
        $b$-weighted single Hurwitz numbers of~\cite{ChapuyDolega2020}
        given by~\eqref{eq:deftaub}, with weight function $G(z)=(1+z)^{-1}$, and with the rescaling  $t
        \mapsto t\cdot u$. Explicitly, it is defined by
	\begin{align}\label{eq:defmain}
                 \tau_b^{Z}(t;\pp,u) &= \sum_{n \geq 0} t^n\sum_{\lambda \vdash n} 
		\frac{J_\lambda^{(b)}(\pp)}{j_\lambda^{(b)}}\prod_{\square
	\in \lambda}\frac{1}{u^{-1}+c_b(\square)},
      \end{align}
      This series is understood as an element of $\mathbb{Q}(b,u)[\pp][[t]]$, and it also belongs to $\mathbb{Q}(b)[\pp][[u,t]]$. 
\end{definition}

\subsection{Evolution equation}
\label{subsec:Evolution}

Theorem 6.1 in~\cite{ChapuyDolega2020}, specialized to our case, shows that the function $\tau_b^{Z}(t;\pp,u)$ obeys a certain linear differential equation involving unbounded iterates of certain explicit differential operators -- corresponding to the fact that the expansion of $G(z)=(1+z)^{-1}$ as an element of $\mathbb{Q}[ [ z] ]$ is infinite. Here we take advantage of the particularly simple form of this function to obtain a much simpler and fully explicit equation.

\begin{theorem}[Evolution equation]\label{thm:evolution}
       The function $\tau_b^{Z}$ is
       uniquely determined by the following equation:
       \begin{align}\label{eq:evolution}
	\frac{td}{dt} \tau_b^{Z}(t;\pp,u) &= E_b \tau_b^{Z}(t;\pp,u),
\end{align}
where
\begin{align}
E_b &= u\left(\frac{tp_1}{1+b} - \left((1+b)\sum_{m,n}p_{m+n}p_m^*p_n^*+ \sum_{n,m \geq 1}
              p_n p_m p^*_{n+m}+b\sum_{n\geq 1}(n-1)\cdot p_{n}p_n^*\right)\right).
\end{align}
      \end{theorem}
      The special case $b=0$ of the above equation is known as the
      cut-and-join equation for the monotone Hurwitz numbers first proved in \cite{GouldenGuayPaquetNovak2013}
      and reproved in \cite{Dunin-BarkowskiKramerPopolitovShadrin2019} using different methods. Our proof reduced to this case gives
      another proof. It is
      based on the theory of Jack polynomials, and 
      has a similar
      flavour as the proof of the decomposition equation
      in~\cite{ChapuyDolega2020}, but it is not the same. The next
      proposition collects properties of Jack polynomials that we will
      need and can be found (explicitly or implicitly) in the seminal
      work of Stanley \cite{Stanley1989}.

\begin{proposition}[\cite{Stanley1989}]
  For any partition $\lambda$ one has
  \begin{align}
    \label{eq:Cauchy}
	  \sum_{n \geq 1} \sum_{\lambda \vdash n} \frac{J_\la^{(1+b)}(\pp)J_\la^{(1+b)}(\qq)}{j_\lambda^{(b)}} &= \sum_{n \geq 1}
                                                      \sum_{\lambda\vdash n}
                                                      \frac{p_\lambda q_\lambda}{z_\la(1+b)^{\ell(\la)}}\\
    \label{eq:JackPieriRule}
    p_1 J_\la^{(1+b)}(\pp) &= \sum_{\la \nearrow \mu}c_{\la \nearrow \mu}
                            J_\mu^{(1+b)}(\pp),
  \end{align}
where $c_{\la \nearrow \mu} \in \Z[b]$ is a (explicit) polynomial in
$b$ with integer coefficients and $\la \nearrow \mu$ means that
$\mu$ is obtained from $\lambda$ by adding a single box.
Moreover, for any partition $\lambda \vdash n$ one has
  \begin{align}
    \label{eq:1^n}
     \langle J_\lambda^{(1+b)}, p_{1^n}\rangle_b &=
     (1+b)^nn!,\\
    \label{eq:21^n}
         \langle J_\lambda^{(1+b)}, p_{21^{n-2}}\rangle_b &=
	 \left(\sum_{\square \in
                                                  \la}c_b(\square)\right)
                                                  (1+b)^{n-1}\cdot 2(n-2)!,\\
        \label{eq:21^nLB}
         D_b J_\lambda^{(1+b)} &=\frac{\langle J_\lambda^{(1+b)},
                                 p_{21^{n-2}}\rangle_b}{(1+b)^{n-1}\cdot
                                 2(n-2)!} J_\lambda^{(1+b)},\\
    \label{eq:CauchySimple}
    \sum_{n \geq 1} \sum_{\lambda \vdash n} \frac{J_\la^{(1+b)}(\pp)}{j_\lambda^{(b)}} &= \sum_{n \geq 1}
                                                      \frac{p_{1^n}}{(1+b)^{n}n!}.
    \end{align}
\end{proposition}
We can now prove the evolution equation.
\begin{proof}[Proof of Theorem~\ref{thm:evolution}]
	Locally to this proof, we write $J_\lambda=J_\lambda^{(1+b)},
        j_\lambda=j_\lambda^{(b)}$ to make the notation lighter. We define also
	\[ \tilde{J}_\lambda(\pp) := \frac{J_\la(\pp)}{j_\la}\prod_{\square
            \in \lambda}\big(u^{-1}+c_b(\square)\big)^{-1}.\]
 Notice that
        \[ \frac{up_1}{1+b}\tilde{J}_\lambda(\pp) = \sum_{\la \nearrow \mu}c_{\la \nearrow \mu}\big(1+uc_b(\mu\setminus\la)\big)
          \frac{j_\mu}{(1+b)j_\la} \tilde{J}_\mu(\pp)\]
        by \cref{eq:JackPieriRule}. Therefore
        \begin{align}
          \label{eq:First}
          [t^{n+1}]\frac{u\cdot t\cdot p_1}{1+b}\tau_b^{Z}(t;\pp,u) =
          \sum_{\lambda \vdash n}\sum_{\mu \vdash n+1}c_{\la \nearrow \mu} \big(1+uc_b(\mu\setminus\la)\big)
          \frac{j_\mu}{(1+b)j_\la} \tilde{J}_\mu(\pp) \\
          = \sum_{\lambda \vdash n}\sum_{\mu \vdash n+1}c_{\la \nearrow \mu}\frac{j_\mu}{(1+b)j_\la} \tilde{J}_\mu(\pp)+u\sum_{\lambda \vdash n}\sum_{\mu \vdash n+1}c_{\la \nearrow \mu}c_b(\mu\setminus\la)\frac{j_\mu}{(1+b)j_\la} \tilde{J}_\mu(\pp). \nonumber 
        \end{align}
        We can rewrite the first sum as follows:
        \begin{align}
          \label{eq:RHS1}
\sum_{\lambda \vdash n}\sum_{\mu \vdash n+1}c_{\la \nearrow
                  \mu}\frac{j_\mu}{(1+b)j_\la} \tilde{J}_\mu(\pp) &=
                  \sum_{\lambda \vdash n}\sum_{\mu \vdash
                  n+1}\left\langle p_1
                  \frac{J_\la}{j_\la},\frac{J_\mu}{1+b}\right\rangle_b
                  \tilde{J}_\mu(\pp) 
                                                                    = \\
                  \sum_{\mu \vdash
                  n+1}\left\langle 
                  \frac{p_1^{n+1}}{(1+b)^nn!},\frac{J_\mu}{1+b}\right\rangle_b
                  \tilde{J}_\mu(\pp) 
                  &= (n+1) \sum_{\mu \vdash
                  n+1}\tilde{J}_\mu(\pp) = [t^{n+1}]\frac{td}{dt}\tau_b^{Z}(t;\pp,u). \nonumber 
        \end{align}
We used \cref{eq:JackPieriRule} in the first equality, then we changed
the order of summation and applied
\eqref{eq:CauchySimple} in the second equality. The last equalities
follow from \eqref{eq:1^n} and from the fact that the exponent in $t$
counts the degree of the associated symmetric functions. Now, consider the second summand. We have
\begin{align}
  \label{eq:RHS2}
u\sum_{\lambda \vdash n}\sum_{\mu \vdash n+1}c_{\la \nearrow \mu}c_b(\mu\setminus\la)\frac{j_\mu}{(1+b)j_\la} \tilde{J}_\mu(\pp) &=
                  u\sum_{\lambda \vdash n}\sum_{\mu \vdash
                  n+1}\left\langle [D_{b},p_1]
                  \frac{J_\la}{j_\la},\frac{J_\mu}{1+b}\right\rangle_b
                  \tilde{J}_\mu(\pp) 
                                                                    = \nonumber \\
                  u\sum_{\mu \vdash
                  n+1}\left\langle 
                  \frac{[D_{b},p_1] p_1^{n}}{(1+b)^nn!},\frac{J_\mu}{1+b}\right\rangle_b
                  \tilde{J}_\mu(\pp) 
                  &= u\sum_{\mu \vdash
                  n+1}\left\langle 
			\big({\textstyle \binom{n+1}{2}-\binom{n}{2}}\big) \frac{p_2p_1^{n-1}}{(1+b)^nn!},J_\mu\right\rangle_b
                    \tilde{J}_\mu(\pp)  \nonumber \\
                  = u\sum_{\mu \vdash
                  n+1}\left\langle \frac{p_2p_1^{n-1}}{(1+b)^n(n-1)!},J_\mu\right\rangle_b
                  \tilde{J}_\mu(\pp) &= [t^{n+1}] 2u D_{b} \tau_b^{Z}(t;\pp,u).
                \end{align}
                Here, the first equality follows from
                the characterization of Jack symmetric functions as
                the eigenfunctions of the Laplace--Beltrami operator and
                from \eqref{eq:JackPieriRule}. The second equality is
                obtained by changing the order of the summation and
                applying \eqref{eq:CauchySimple}. The last equalities
                are obtained by a direct computation of the action of
                the Laplace--Beltrami operator on the power-sum
                symmetric functions and by applying \eqref{eq:21^n}
                and \eqref{eq:21^nLB}.

                Collecting \eqref{eq:RHS1} and \eqref{eq:RHS2} and
                comparing them with \eqref{eq:First} we obtain
                \[ \frac{u\cdot t\cdot p_1}{1+b}\tau_b^{Z}(t;\pp,u)
                  =\left(\frac{td}{dt}+2u D_{b}\right)\tau_b^{Z}(t;\pp,u).\]
It gives directly \eqref{eq:evolution} by using the formula~\eqref{eq:Laplace-Beltrami} for $D_b$.
        \end{proof}

\subsection{Virasoro constraints}

We now prove the Virasoro constraints, which can be viewed as a refinement of the evolution equation~\eqref{eq:evolution} (the evolution equation is the sum of the Virasoro constraints weighted by $p_i$).
In classical combinatorial models (such as maps and bipartite maps)
and for $b\in\{0,1\}$, the evolution equation has a simple
combinatorial interpretation as a root-deletion procedure, and this
combinatorial proof of the evolution equation in fact directly proves
the Virasoro constraints. For arbitrary $b$, things are more complicated (see also~\cite[Rem.~5]{ChapuyDolega2020}) and we must proceed differently. The next lemma, which can be used in other models (see Appendix~\ref{sec:otherModels}) says that under mild assumptions one can in fact go backwards and deduce the Virasoro  constraints from their weighted sum.

\begin{lemma}[Virasoro constraints from their sum]\label{lemma:evolutionImpliesVirasoro}
        Suppose that $R$ is a ring and let $F \in R[\pp][[t,t_1,\dots,t_r]]$
        be a formal power series
	such that for each $n\geq 0$,
		$[t^n]F$ is a homogenous
        polynomial in $\pp$ of degree $n$, where $\deg(p_i) := i$. 
	Suppose
        that there exists an operator $A$ such that
        \begin{itemize}
        \item $(\frac{td}{dt}+s^kA)F = 0$ for a positive integer $k$,
          with $s \in \{t,t_1,\dots,t_r\}$ and $A$ independent of $s$;
          \item $p_i^* F\big|_{s=0} = 0$ for each positive integer $i$;
          \item $[t^n] L_i F$ is a homogenous polynomial in $\pp$ of degree $n-i$;
          \item $A = \sum_{i \geq 1}p_i \tilde{L}_i$, where the operators $L_{i-k\delta_{s,t}} := \frac{p_i^*}{s^k}+\tilde{L}_i$ represent  the Virasoro algebra, i.e. $[L_i,L_j] =
            (i-j)L_{i+j}$.
            \end{itemize}
            Then, the function $F$ satisfies the following Virasoro constraints:
	    \[L_iF = 0,\]
            for each $i \geq 1-k\delta_{s,t}$.
          \end{lemma}

We directly obtain 

\begin{theorem}[Virasoro constraints for $b$-monotone Hurwitz numbers]\label{thm:virasoro}
	The function $\tau_b^{Z}$ satisfies the following Virasoro constraints, for $i \geq 1$:
	\begin{align}\label{eq:virasoro}
L^{Z}_i \tau_b^{Z}(t;\pp,u) = 0,
\end{align}
with
\begin{align}
  \label{eq:defLZi}
         L^{Z}_i &:= \frac{p_{i}^*}{u}+\bigg((1+b)\sum_{m+n=i}p_m^*p_n^*+ \sum_{n \geq 1}
               p_n p^*_{n+i}+b(i-1)p_i^*-\frac{t\delta_{i,1}}{(1+b)}\bigg).
\end{align}
      \end{theorem}
      \begin{proof}[Proof of Theorem~\ref{thm:virasoro}]
	      The fact that the family of operators  $\{L^{Z}_i, i\geq 1\}$ satisfy Virasoro relations  is a direct check. The statement is then a consequence of the evolution equation (Theorem~\ref{thm:evolution}) and Lemma~\ref{lemma:evolutionImpliesVirasoro} applied to
	      $F=\tau_b^{Z}, s=u, k=1$.
\end{proof}
The reader is refered to Appendix~\ref{sec:otherModels} to see applications of Lemma~\ref{lemma:evolutionImpliesVirasoro} with other values of $s$ and $k$.
It remains to prove the lemma.
	  \begin{proof}[Proof of Lemma~\ref{lemma:evolutionImpliesVirasoro}]
            In order to show that $L_iF = 0$ we will prove by induction on $n\geq 0$ that for each $i \geq
            1-k\delta_{s,t}$ the coefficient $[s^n]\widehat{L}_iF = 0$,
            where $\widehat{L}_i := s^kL_i$. For $n=0$
            one has
            \[ [s^0]\widehat{L}_iF = [s^0]p_i^*F = p_i^*
              F\big|_{s=0} = 0,\]
            by assumption. Fix a  positive integer $n$ and assume that for each $i \geq
            1-k\delta_{s,t}$ and for each $l < n$ the coefficient $[s^l]\widehat{L}_iF = 0$.
            We have that
            \[ \left[\frac{td}{dt}+s^kA, \widehat{L}_i\right]F =
              \left(\frac{td}{dt}+s^kA \right) \widehat{L}_i F.\]
            On the other hand, the commutator on the LHS can be evaluated as follows. First, since $[t^n]\widehat{L}_i F$ is a homogeneous polynomial of degree $n-i$ in $\pp$,
            \[\left[\frac{td}{dt},\widehat{L}_i\right] F = i\widehat{L}_i F + \left[\sum_{n\geq 1} p_n p_n^*, \widehat{L}_i\right]F\]
            therefore
            \begin{equation*}
            \begin{aligned}
             \left[\frac{td}{dt}+s^kA, \widehat{L}_i\right]F &= \sum_{j
                \geq 1} [s^{k} p_j L_{j-k\delta_{s,t}}, s^{k}L_i]F + i\widehat{L}_i F\\
                & = s^{2k}\bigg(\sum_{j \geq 1} \big(p_j
              [L_{j-k\delta_{s,t}},L_i] +
              [p_j,L_i]L_{j-k\delta_{s,t}}\big)\bigg)F + i\widehat{L}_i F.
              \end{aligned}
              \end{equation*}
            Substituting $[L_j,L_i] =
            (j-i)L_{i+j}$ and $ [p_j,L_i] =
            \frac{-(i+k\delta_{s,t})\delta_{j,i+k\delta_{s,t}}}{s^k}+O_s(1)$, we end up with
            \[ \left[\frac{td}{dt}+s^kA, \widehat{L}_i\right]F  = s^k\sum_{j
                \geq 1} \big(p_j
              (j-i-k\delta_{s,t})\widehat{L}_{i+j-k\delta_{s,t}}+B\widehat{L}_{j-k\delta_{s,t}}\big)F-k\delta_{s,t}
              \widehat{L}_iF, \]
            where $B$ is an operator satisfying $B = O_s(1)$.
            Therefore, we have the following equality
            \[ \left(\frac{td}{dt}+s^kA \right) \widehat{L}_i F = s^k\sum_{j
                \geq 1} \big(p_j
              (j-i-k\delta_{s,t})\widehat{L}_{i+j-k\delta_{s,t}}+B\widehat{L}_{j-k\delta_{s,t}}\big)F-k\delta_{s,t}
              \widehat{L}_iF.\]
            We now compare the coefficients of $s^{n}$ on
            the both side of the equation. Using our induction
            hypothesis and the fact that
            $[s^{n}]\frac{td}{dt}\widehat{L}_iF = n\delta_{s,t} [s^{n}]\widehat{L}_iF+\frac{td}{dt}[s^{n}]\widehat{L}_iF$ we have the following identity:
            \[ \frac{td}{dt}[s^{n}]\widehat{L}_iF = -(k+n)\delta_{s,t}
              [s^{n}]\widehat{L}_iF,\]
            which implies that
            $[s^{n}]\widehat{L}_iF = 0.$
		  This concludes the induction step.
          \end{proof}

\begin{remark}\label{rem:MMS11}
	The Virasoro constraints of \cref{thm:virasoro} have already
          appeared in the work of Mironov, Morozov and
          Shakirov~\cite{MironovMorozovShakirov2010} as a special limit, called ``pure gauge'' limit, of a set of differential equations written by Kaneko \cite{Kaneko1993}. Those differential equations have furthermore been shown to be solved by some $\beta$-deformation of a Selberg integral. The pure gauge limit of this integral could thus provide an integral representation of our function $\tau^Z_b$. Note
          that our proof of the Virasoro constraints, which is combinatorial in nature, deals directly with Jack symmetric functions
          and does not rely on matrix integrals, thus answering a question from \cite{MironovMorozovShakirov2010}.
\end{remark}

          \subsection{$b$-deformed contents and Jack characters}

          Let $\mathcal{C}_b(\lambda) := \{c_b(\square)\colon \square
          \in \lambda\}$ denote the multiset of $b$-deformed contents
          and recall that (non-normalized) \emph{Jack characters}
          $\theta_\mu(\lambda) := [p_\mu]J_\lambda$ are the coefficients
          of Jack symmetric functions expanded in the power-sum
          basis (see
          \cite{DolegaFeray2016} for the explanation of this terminology).
		  It was proved by F\'eray~\cite[Proposition
          4.1]{Feray2012} that $\{\theta_\mu(\cdot)\}_{\mu \vdash n}$
          is a linear basis (over $\mathbb{Q}(b)$) of the space of
          functions $f\colon \mathcal{P}_n \to \mathbb{Q}(b)$ on the
          set $\mathcal{P}_n$ of partitions of size $n$. In particular,
          one can express the complete homogenous functions evaluated in
            $b$-deformed contents as a linear combination of Jack characters:
          \[ h_k(\mathcal{C}_b(\lambda)) = \sum_{\mu \vdash
              |\lambda|}a^k_\mu\theta_\mu(\lambda).\]
          We now prove a conjecture of Féray showing that the coefficients $a^k_\mu$ have a recursive structure.

	  \begin{theorem}[{\cite[Conjecture 4.3]{Feray2012}}]\label{thm:Feray}
            The coefficients $a^k_{\rho}$ fulfill the linear relation:
             \begin{equation}
              \label{eq:Feray}
              a^k_{\rho \cup (m)} = \delta_{m,1}a^k_{\rho}+\sum_{r+s=m}a^{k-1}_{\rho \cup (r,s)}+(1+b)\sum_{i=1}^{\ell(\rho)}\rho_i a^{k-1}_{\rho\setminus{\rho_i} \cup (\rho_i+m)}+b(m-1) a^{k-1}_{\rho \cup (m)},
              \end{equation}
			for any $m \geq 1$,
where $\rho\setminus m$ (respectively $\rho\cup (m)$) means that a row of size $m$ is removed from (respectively added to) $\rho$.
            \end{theorem}

            \begin{proof}
              Going back to the \cref{def:mainfunction} of $\tau^Z_b$, we expand $\prod_{\square
	\in \lambda}\frac{1}{u^{-1}+c_b(\square)}$ onto complete homogeneous symmetric functions, so that
              \[ \tau^Z_b = \sum_{n \geq 0} (ut)^n\sum_{\lambda \vdash n} 
		\frac{J_\lambda^{(b)}(\pp)}{j_\lambda^{(b)}}\sum_{k
                  \geq 0}(-u)^k h_k(\mathcal{C}_b(\lambda)).\]
            It is further rewritten as
              \[ \tau^Z_b = \sum_{n \geq 0} (ut)^n\sum_{\lambda,\rho,\mu \vdash n} 
		\frac{\theta_\rho(\lambda)\theta_\mu(\lambda)}{j_\lambda^{(b)}}p_\rho\sum_{k
                  \geq 0}(-u)^k a^k_\mu(\lambda)\]
              by expanding Jack symmetric functions in the power-sum
              basis and using the definition of $a^k_{\rho}$. We now extract coefficients starting with
                            \[ [p_\rho]\tau^Z_b = (ut)^{|\rho|}\sum_{\lambda,\mu \vdash n} 
		\frac{\theta_\rho(\lambda)\theta_\mu(\lambda)}{j_\lambda^{(b)}}\sum_{k
                  \geq 0}(-u)^k a^k_\mu= (ut)^{|\rho|}\frac{1}{z_\rho(1+b)^{\ell(\rho)}}\sum_{k
                  \geq 0}(-u)^k a^k_\rho,\]
              where the last equality follows the orthogonality relation
              \[ \sum_{\lambda, \vdash n} 
		\frac{\theta_\rho(\lambda)\theta_\mu(\lambda)}{j_\lambda^{(b)}}
                = \frac{\delta_{\mu,\rho}}{z_\rho(1+b)^{\ell(\rho)}}\]
              which is a consequence of the orthogonality of the
              power-sum symmetric functions and Jack symmetric
              functions with respect to $\langle,\rangle_b$.
              Therefore, the coefficient $a^k_\rho$ can be obtained via the following coefficient extraction
			\[ a^k_\rho = [(ut)^{|\rho|}(-u)^kp_\rho](1+b)^{\ell(\rho)}z_\rho\tau^Z_b.\]
			Relation \eqref{eq:Feray} is then an immediate consequence of the Virasoro
              constraint
              $L^{Z}_m \tau^Z_b = 0.$
            \end{proof}

            \begin{remark}
              F\'eray suggested that the quantities $a^k_{\rho}$ might
              be interpreted combinatorially, which is known for
              $b=0,1$ due to the connection with Jucys-Murphy elements
              \cite{Novak2010,Zinn-Justin2010,Matsumoto2010}. In the following section we
              provide such a combinatorial interpretation.
              \end{remark}

\section{Graphical model: Monotone Hurwitz Maps}
\label{sec:model}

In this section we describe a combinatorial model of graphs whose
partition function is $\tau_{b}^{Z}(t;\pp,u)$ (up to a minor rescaling of variables).
Note that \cite[Section
6]{ChapuyDolega2020} provides such a model for general
constellations\footnote{General constellations from
  \cite{ChapuyDolega2020} give a combinatorial model for double
  $b$-deformed (weighted)
  Hurwitz numbers, while here we are working with single Hurwitz
  numbers. In the classical case $b=0$ double monotone Hurwitz numbers
  give a $1/N$-expansion of the HCIZ integral and we asked in
  \cite[Section~6.5]{ChapuyDolega2020} if our combinatorial expansion
  extends this result and gives a $1/N$-expansion of the $\beta$-deformation of HCIZ integral
  introduced in \cite{BrezinHikami2003}. The answer for this question
  is affirmative, which follows from the expansion of the
  $\beta$-HCIZ integral in Jack symmetric functions derived in~\cite{HikamiBrezin2006}.}
which can be specialized to our case, but the model we describe here
is arguably simpler and more explicit in our case. It also covers the
model of~\cite{GouldenGuayPaquetNovak2013}, which agrees with our
model at $b=0$ via classical encoding of
permutations factorizations into transpositions by Hurwitz maps~\cite{Moszkowski1989,Poulalhon1997,OkounkovPandharipande2009}.
In the following definition, we allow maps on the sphere with one
vertex, one face and no edges (conventionnally viewed as a 2-cell embedding of the one-vertex graph on the sphere).
\begin{definition}[Hurwitz map]
	A labelled \emph{Hurwitz map} $M$ with $n$ vertices and $r$ edges is a 2-cell embedding of a loopless multigraph on a compact surface, with the following properties: 
        \begin{itemize}[itemsep=0pt, topsep=0pt,parsep=0pt, leftmargin=12pt]

		\item [(a)] the vertices of the map are labelled from $1$ to $n$, and the neighbourhood of each vertex is equipped with an orientation. Moreover each vertex has a distinguished corner called \emph{active}, and {\it we let $c_i$ be the active corner incident to the vertex $i$}. The corner $c_i$ is most easily represented as an arrow pointing to $i$ in the corresponding angular sector.

		\item [(b)] the edges of the map are labelled from $1$ to $r$; {\it We let $e_i$ be the edge labelled $i$ and $M_i$ be the submap\footnote{given a subset of edges $E$ in a map $M$, the submap of $M$ induced by $E$ is best pictured by viewing $M$ as a ribbon graph, and taking the ribbon graph with the same vertex set as $M$ but keeping only edges (ribbons) in $E$, keeping all the topological incidences between these edges and the vertices; it possibly lives on a different surface than the map $M$ itself.} of $M$ induced by edges $e_1,e_2,\dots,e_i$}.

		\item [(c)] for each $i$ in $[1,r]$, let $a_i<b_i$ be the two vertices incident to the edge $e_i$. Then in the map $M_{i}$, the following is true: 
	
        \begin{itemize}[itemsep=0pt, topsep=0pt,parsep=0pt, leftmargin=24pt]
		\item [(c1)] In the local orientation around vertex $b_i$, the active corner  $c_{b_i}$ immediately follows the edge $e_i$ (see Figure~\ref{fig:hurwitzMap}-Left); 
		\item [(c2)] the corner which is opposite to $c_{b_i}$ with respect to the edge $e_i$ is the active corner $c_{a_i}$ of vertex $a_i$ (see Figure~\ref{fig:hurwitzMap}-Left);
		\item [(c3)] if the edge $e_i$ is disconnecting in $M_i$, then the orientations of the vertices $a_i$ and $b_i$ are compatible in $M_i$ (i.e. they can be jointly extended to a neighbourhood of $e_i$).
	\end{itemize}
\end{itemize}
	If moreover one has $b_1\leq b_2\dots \leq b_r$, then the map is called \emph{monotone}.

	The \emph{degree} of a face is its number of active corners. These face degrees form a partition of $n$ called the \emph{degree profile} of $M$. 
\end{definition}

\begin{figure}
	\centering
	\includegraphics{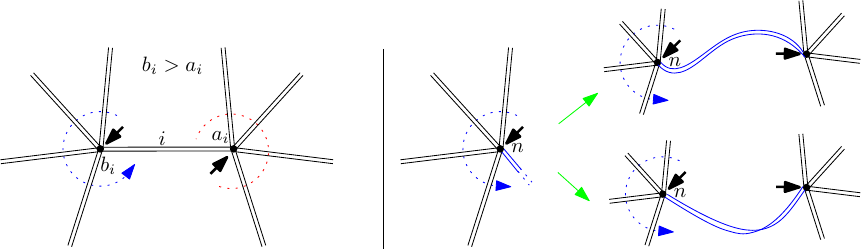}
	\caption{Left: the local constraints around the edge $e_i$ in a Hurwitz map. Edges of label $>i$ may be incident to the vertices $a_i$ or $b_i$ but they are not represented on this picture. Active corner are represented by arrows.
	Right: when we want to attach a new edge of maximal label to the vertex $n$, there is a unique possible corner of attachment around this vertex. Given any other active corner of the map, there are a priori two choices of attachment of the other end of the edge, corresponding to the two sides of the arrow, each corresponding to a given twist of the edge. When this new edge joins two connected components together, property (c3) asserts that only one of these two choices is valid.}\label{fig:hurwitzMap}
\end{figure}

	We stress that the local properties (c) in the above definition hold in the map $M_i$, not $M$. That is to say, for a given $i$, edges of label greater than $i$ play no role in this constraint.

\begin{remark}
	When the underlying surface is oriented, the constraint (c) is
        equivalent to the fact that in the map $M=M_r$, vertices are
        oriented according to surface orientation and that the edge
        labels around each vertex are increasing, starting from the
        active corner. This is easily seen using (c) and induction on
        $i$. Such maps are often called \emph{Hurwitz maps} in the map
        community, we thus give here an non-oriented generalization of
        this notion. Oriented Hurwitz maps are in bijection with
        tuples of transpositions by viewing the edge $e_i=(a_i,b_i)$
        as a transposition in $\mathfrak{S_n}$. In this
        correspondence, the sequence of active corners around faces of
        the map correspond to the cycles of the permutation
        $\phi=e_1\dotsm e_r$. Hence the cycle type of $\phi$ matches
        the degree profile of $M$, which is why these maps are
        conveniently used as a topological model for (monotone or not)
        Hurwitz numbers (see~\cite{OkounkovPandharipande2009}).  
\end{remark}

The \emph{canonical decomposition} of a monotone Hurwitz map $M$ is
the following algorithm:
\begin{itemize}
  \item[(1)] if the vertex of maximum label is isolated, erase it; if
    not, remove the edge of maximum label (note that it has to be
    incident to that vertex by monotonicity);
  \item[(2)] iterate until no vertex remains.
    \end{itemize}

As in~\cite{ChapuyDolega2020}, every time an edge is deleted by the algorithm, we collect a certain \emph{$b$-weight} in $\{1,b\}$ subject to the constraints of Measures of Non Orientability (MON). We refer the reader to that paper for details about MON. Let us just say that a MON is a way to associate a weight to each edge deletion which depends on the topological relation between the edge and the map it is deleted from. In particular, when deleting an edge $e$ results in a map $M$, the associated $b$-weight is equal to (Figure~\ref{fig:MON}):
\begin{itemize}
	\item $1$ if $e$ joins two different connected components of $M$, or if $e$ splits a face of $M$ into two distinct faces;
	\item $b$ if $e$ is a twisted diagonal added inside a face of $M$;
	\item $1$ or $b$ if $e$ joins two distinct faces inside the same connected component of $M$. Moreover, the $b$-weight associated to $e$ is $1$ if and only if the $b$-weight associated to its twisted edge $\tilde{e}$ is $b$. Finally, the $b$-weight is always $1$ when $M\cup \{e\}$ is orientable.
\end{itemize}
Moreover, we also ask the $b$-weight to depend only on the connected component in which the edge is deleted. Figure~\ref{fig:MON} summarizes what we need to know about measures of non-orientability and $b$-weights. 
The product of all $b$-weights collected by the canonical
decomposition is a monomial of the form $b^{\nu(M)}$, where $\nu(M)$
is an integer associated with the Hurwitz map $M$. Note that $\nu(M)=0$ if and only if $M$ is orientable.

Finally, we let $cc(M)$ be the number of connected components of $M$.

\begin{proposition}
The series $\tilde \tau_{b}^{Z}(t;\pp,u):= \tau_{b}^{Z}(-t/u;\pp,-u)$ is the generating function of monotone labelled Hurwitz maps in the sense that
	\begin{align}\label{eq:partitionFunction}
		\tilde \tau_{b}^{Z}(t;\pp,u) = \sum_{n\geq 0} \frac{t^n}{(1+b)^{cc(M)}n!} 
		\sum_{\ell \geq 0} u^\ell \sum_{M\in \mathcal{M}(n,\ell)} b^{\nu(M)} p_{\lambda(M)},
	\end{align}
	where $\mathcal{M}(n,\ell)$ is the set of monotone labelled Hurwitz maps with $n$ vertices and $\ell$ edges, and where $\lambda(M)\vdash n$ is the degree profile of $M$.
\end{proposition}
This combinatorial interpretation remains valid at the level of
connected objects. Namely, the function $(1+b) \ln
\tilde\tau_{b}^{Z}(t;\pp,u)$ has an expansion similar to
\eqref{eq:partitionFunction}, where the sum in the RHS is restricted
to connected maps, and without the weight
$\frac{1}{(1+b)^{cc(M)}}$. Moreover, the function  $(1+b) \tfrac{td}{dt} \ln \tilde\tau_{b}^{Z}(t;\pp,u)$ has coefficients in $\mathbb{N}[b]$. This follows from the evolution equation of \cref{thm:evolution} written for $(1+b) \tfrac{td}{dt} \ln \tilde\tau_{b}^{Z}(t;\pp,u)$, by induction on the order of $u$.
This was shown already in~\cite{ChapuyDolega2020} in larger generality but with a different combinatorial interpretation. The coefficients of this function can naturally be called the $b$-deformed non-oriented monotone single Hurwitz numbers.
\begin{figure}
	\centering
	\includegraphics[width=\linewidth]{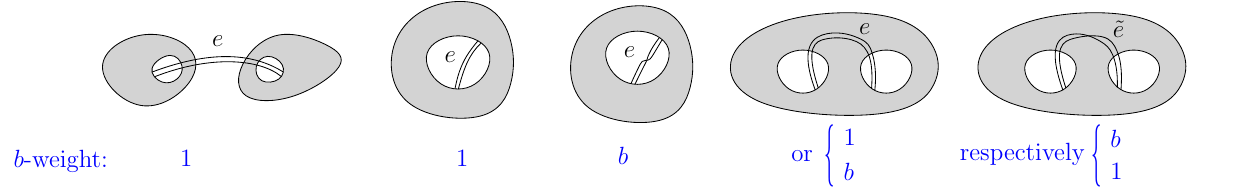}
	\caption{How the $b$-weight is computed. In the last two cases, the weights of $1,b$ or $b,1$ are decided arbitrarily in general, but if one of the two pictures is orientable, then it must get the weight $1$.}\label{fig:MON}.
\end{figure}

\begin{remark}\label{rem:scalings}
	The reader may wonder why we chose to work with the function $\tau_{b}^{Z}(t;\pp,u)$ since the rescaled function $\tilde\tau_{b}^{Z}(t;\pp,u)$ has nonnegative coefficients and a more natural combinatorial interpretation. This (debatable) choice may become clearer in the next section where $u$ will play the role of an inverse dimension parameter.
\end{remark}

\begin{proof}
	The proof of \cite{GouldenGuayPaquetNovak2013} in the orientable case, which interprets the evolution equation~\eqref{eq:evolution} (for $b=0$) as the deletion of the edge $e_r$, relies on a centrality property of symmetric functions of Jucys-Murphy elements which is not obvious to extend to the non-orientable (let alone, $b$-deformed) case. 
	We thus need to use a different approach in which the full vertex of maximal label is removed, rather than a single edge. To do this, we use~\cite[Eq.~(61)]{ChapuyDolega2020} which reads, specialized to our case
\begin{align}\label{eq:CDeq61}
 \frac{d}{dt} 
	\tilde \tau = \Theta_Y Y_+ \frac{1}{1-u\Lambda_Y} \frac{y_0}{1+b} \tilde 
\tau,
\end{align}
where $\Theta_Y$ and $Y_+$ are the operators that substitute $y_i$ to $p_i$ and to $y_{i+1}$, respectively, and where
$$
\Lambda_Y := (1+b)\sum_{i,j\geq 1}y_{i+j-1}\frac{i\partial^2}{\partial
  p_i \partial y_{j-1}} +\sum_{i,j\geq
  1}y_{i-1}p_j\frac{\partial}{\partial y_{i+j-1}}
+b\cdot\sum_{i\geq
  0}y_{i}\frac{i\partial}{\partial y_i }.
$$

	In the rest of the proof we use the notation $\tilde \tau$ for the RHS of~\eqref{eq:partitionFunction} and we will prove~\eqref{eq:CDeq61}. This will be enough to conclude since this equation characterizes coefficients of $\tilde\tau$, inductively.

	We assume that~\eqref{eq:CDeq61} holds up to order
        $t^{n-1}$. Every monotone Hurwitz map with $n$ vertices can be
        constructed from one of size $n-1$ by adding the vertex $n$
        and all edges incident to it, with increasing edge labels. We now
        analyse this construction at the level of generating
        functions; during this analysis we use the variables $y_i$
        rather than $p_i$ to mark the degree of the face containing
        the active corner of $n$, and we do not count the contribution of this corner
        to the face degree (this is similar to~\cite{ChapuyDolega2020}).
The process of adding the vertex $n$ goes as follows:

\begin{itemize}[itemsep=0pt, topsep=0pt,parsep=0pt, leftmargin=24pt]
	\item[(i)] We create the isolated vertex $n$. The contribution to the generating function is $y_0$ (since the active corner $n$ is unique in its face but we do not count it by convention) times $1/(1+b)$ (since a new connected component is created) times $n[t^{n-1}]\tilde \tau$ (since we choose a map of size $n-1$ but need to take into account the factor $\frac{n!}{(n-1)!}$ coming from generating functions). 
	\item[(ii)] We add a certain number (say $k\geq0$) of edges, in increasing label, to the vertex $n$. The contribution of this step is analyzed below.
	\item[(iii)] We finally account for the contribution of the active corner of label $n$ in its face degree (operator $Y_+$) and we restore the $p$-variable weighting for that face (operator $\Theta_Y$).
\end{itemize}
		We now analyse the contribution of each iteration in Step (ii). Assume that the current map has $i-1$ edges, so we want to add the edge $e_i$ to it. By property (c1), we need to attach this edge to the vertex $n$ just before the arrow materializing the active corner (in local orientation). Now, there are several ways to attach the other end of the edge $e_i$:

\begin{itemize}[itemsep=0pt, topsep=0pt,parsep=0pt, leftmargin=36pt]

	\item we attach to a corner in the same face, by splitting it into two faces. By property (c2), for each active corner in the face there is a unique way to attach $e_i$ to it in this way, thus splitting the degree of the face into two parts. The contribution to the generating funtion is thus $\sum_{i,j\geq
  1}y_{i-1}p_j\frac{\partial}{\partial y_{i+j-1}}$.

	\item we attach to a corner in the same face, twisting the edge in order to maintain the same number of faces. Again by property (c2), for each active corner in the face there is a unique way to attach $e_i$ in this way; the face degree is unchanged and a weight $b$ is collected in the operation; so the corresponding operator is $b\cdot\sum_{i\geq
  0}y_{i}\frac{i\partial}{\partial y_i }$.

	\item we attach to a corner in a different face, thus merging that face and the current face together and adding degrees. Each choice of an active corner in another face of the map a priori gives two ways to add such an edge. Indeed, we can connect $e_i$ to either side of the arrow corresponding to the active corner chosen, and for each choice there is a unique way to twist the edge $e_i$ so that (c2) holds (Figure~\ref{fig:hurwitzMap}-Right). These two choices have the same impact on face degrees but they correspond to contributions of $1$ and $b$ to the $b$-weight. However, this analysis has to be corrected when the chosen corner belongs to a different connected component than $n$. In that case, only one of the two choices is valid by (c3), but since the number of connected components decreases by one the factor $\frac{1}{(1+b)^{cc(M)}}$ is multiplied by $(1+b)$. Therefore the two cases (same component or not) give a similar contribution and overall they are taken into account by the operator
$(1+b)\sum_{i,j\geq 1}y_{i+j-1}\frac{i\partial^2}{\partial p_i \partial y_{j-1}}$.
\end{itemize}
Overall, we see that the contribution of all cases is precisely given by the operator $\Lambda_Y$. The total contribution of Step (iii) is therefore given by the operator $\sum_{k\geq0} (u\Lambda_Y)^k=\frac{1}{1-u\Lambda_Y}$, and this concludes the proof.
\end{proof}

\begin{remark}\label{rem:classicalHurwitz}
	The notion of Hurwitz maps that we introduced here is of independent interest, even without the notion of monotonicity. The (non-necessarily monotone) Hurwitz maps give a combinatorial model for the $b$-deformed tau function of classical single Hurwitz numbers, i.e. the function $\tau_{b}^{G}(t;\pp,\qq=(\delta_{i,1}),u)$ with $G(z)=e^z$ introduced in~\cite{ChapuyDolega2020}. To see this, it is better to work directly with the cut-and-join equation of this model~\cite[Eq~(67)]{ChapuyDolega2020} (with $q_i=\delta_{i,1}$ and $\hbar=u$ in the notation of this paper). Indeed it can be interpreted as describing the deletion of the edge of largest label in a Hurwitz map, similarly as what we did here. We leave details to the reader.
      \end{remark}

We conclude this section by noting that the coefficients of the function $\tau_{b}^{Z}(t;\pp,u)$, being a special case of the weighted-Hurwitz numbers of~\cite[Section 6]{ChapuyDolega2020}, have an interpretation as counting (with weights) certain generalized branched coverings of the sphere by surfaces (orientable or not). It is natural to suspect that, in the same way as we were able in this section to replace in our special case the general model of constellations of~\cite{ChapuyDolega2020} by a much simpler model (monotone Hurwitz maps), it is possible to obtain an explicit model of generalized  branched coverings counted by $b$-monotone Hurwitz numbers which would be simpler than what results from directly specializing the definitions of~\cite{ChapuyDolega2020}. However, this would lead us too far from our main subject and we will not adress this question in this paper. 

\section{Schur expansion for $b=1$ and orthogonal group characters}
\label{sec:schurExpansion}

In this and the next sections, we consider the case $b=1$ for which Jack polynomials coincide with the \emph{zonal polynomials}, 
$$
Z_\lambda(\pp) := J_\lambda ^{(b)}(\pp) \Big|_{b=1}.
$$
They are zonal spherical functions for the Gelfand pair
$(\Gl(N),O(N))$, which explains the terminology (see~\cite[Chap.~7]{Macdonald1995}).

The purpose of this section is to show that the function
$\tau_{b=1}^{Z}(t;\pp,u)$ defined in~\eqref{eq:defmain} by its
expansion in zonal polynomials, has in fact an explicit expansion in
Schur functions, provided we scale the variables $\pp$ by a factor of
$2$. Moreover, this explicit expansion is directly related to
irreducible representations of orthogonal groups.  See
Theorem~\ref{thm:schurExpansion} and Theorem~\ref{thm:OliveiraNovaes}
which establishes the closely related conjecture of Oliveira and Novaes.

\subsection{Irreducible representations of orthogonal and special
  orthogonal groups.}
The representation theory of orthogonal and special orthogonal
 groups was developed in the pionering work of Weyl \cite{Weyl1939}, see e.g. \cite{FultonHarris1991} for an introduction. The highest weight irreducible characters $o_\lambda$ of the orthogonal group
  $O(2n)$ are indexed by partitions $\lambda$ with $\ell(\lambda) \leq  n$. For $\lambda_n=0$ they coincide with the irreducible characters
  $so_\lambda$ of the special orthogonal group
  $SO(2n)$, and when $\lambda_n \neq 0$ the restriction of $o_\lambda$
  to $SO(2n)$ splits into a direct sum
  $so_{(\lambda_1,\dots,\lambda_n)}+so_{(\lambda_1,\dots,-\lambda_n)}$. The
  dimension of the irreducible representation of $SO(2n)$ of the
  highest weight $\lambda$ is given by the following 
  formula due to Weyl (see \cite{Weyl1939,FultonHarris1991}), which is a
  consequence of the Weyl character formula:
       \begin{align}
        \label{eq:dimWeyl}
        so_\lambda(1^{2n}) &=
                                                           \prod_{1
							   \leq i < j
                             \leq
                             n}\frac{(\rho_i-\rho_j)(\rho_i+\rho_j+2n)}{(-i+j)(-i-j+2n)},      \end{align}
                         where $\rho_i := \lambda_i-i$.
   In particular        \begin{align}
        \label{eq:dimOrthogonal}
        o_\lambda(1^{2n}) &= \begin{cases} so_\lambda(1^{2n}) &\text{
            for } \ell(\lambda) < n,\\ 2so_\lambda(1^{2n}) &\text{
            for } \ell(\lambda) = n.\end{cases}\end{align}
 A different formula for the dimension of the irreducible
representation of the orthogonal group $O(n)$ of the highest
weight $\lambda$, valid regardless of the parity of $n$, was given by El Samra and King \cite{ElSamraKing1979}. They proved that
	   \begin{align}
        \label{eq:dimWeyl2}
        o_\lambda(1^{n}) &= \frac{1}{\hook_\lambda}
                                                           \prod_{\substack{(x,y)
                           \in \lambda\colon\\ x \leq y}}(n+\lambda_x+\lambda_y-x-y) \prod_{\substack{(x,y)
                           \in \lambda\colon\\ x >
         y}}(n-\lambda^t_x-\lambda^t_y+x+y-2),    \end{align}
where $\hook_\lambda$ is the \emph{hook-product} of the partition $\lambda$, which can be defined in previously introduced Jack-theoretic notation by
      $$\hook_\lambda := \hook_{b=0}(\lambda) =\hook_{b=0}'(\lambda) .
      $$
The  quantity $\hook_\lambda$ has a well known representation theoretic
      interpretation, namely $f^\lambda:=|\lambda|!/\hook_\lambda$ is the
      dimension of the irreducible representation of the symmetric
      group indexed by $\lambda$, see e.g.~\cite{Stanley:EC2,
	   FultonHarris1991}. %

	   From \eqref{eq:dimWeyl2}, the quantity $o_\lambda(1^{n}) \in \mathbb{Q}[n]$ can be considered
   as a
   polynomial in $n$ of degree $|\lambda|$, which allows us to extend
   the definition of this ``dimension'' to the case where $n$ is a
   formal variable. Below we will use this convention with
   $n=u^{-1}$, so that $o_\lambda(1^{u^{-1}})\in \mathbb{Q}[u^{-1}]$
   and therefore
$1/o_\lambda(1^{u^{-1}})$ has a valid power series expansion at $u=0$, i.e. $1/o_\lambda(1^{u^{-1}}) \in \mathbb{Q}[[u]]$.

\subsection{Schur function expansion}

In what follows, we will sometimes need to rescale the variables of Schur functions by a factor of~$2$, and we will use the notation
$$
\mathbf{p/2} := (p_1/2, p_2/2, p_3/2, \dots) 
\ \ , \  \ 
\mathbf{2 p} := (2 p_1, 2p_2, 2p_3, \dots). 
$$

\begin{theorem}[Explicit expansion of $\tau_{b=1}^{Z}$ in scaled Schur functions]\label{thm:schurExpansion}
	The function  $\tau_{b=1}^{Z}(t;\pp,u)$ defined by its expansion~\eqref{eq:defmain}, has the following expansion in Schur functions of the variables~$\mathbf{p/2}$,
	\begin{align}\label{eq:schurExpansion}
		\tau_{b=1}^{Z}(t;\pp,u) =\sum_{n\geq 0} t^n \sum_{\lambda}\frac{s_\lambda(\mathbf{p/2})}{\hook_\lambda^2\cdot
		o_\lambda(1^{u^{-1}})}.
	\end{align}
	This identity holds in $\mathbb{Q}(u)[\pp][[t]]$ and in $\mathbb{Q}[\pp][[u,t]]$.
      \end{theorem}

The proof of Theorem~\ref{thm:schurExpansion} consists in showing that the RHS of~\eqref{eq:schurExpansion} satisfies the evolution equation~\eqref{eq:evolution}. In fact, we will directly check each Virasoro constraint since this is not more difficult.
For a partition $\lambda$ and an integer $k$ such that $\ell(\lambda) \leq k$, we define the rational function
$$a_\lambda(n) := \frac{1}{\hook_\lambda^2\cdot
	o_\lambda(1^{2n})} \in \mathbb{Q}(n).
$$
The next lemma tells us how to evaluate this rational function on integers. Here and later we use the notation $\rho_i:=\lambda_i-i$, with the convention $\lambda_i:=0$ if $i>\ell(\lambda)$.
\begin{lemma}
	 If $n > k \geq \ell(\lambda)$, and if $n= k = \ell(\lambda)$, we have 
	\begin{align}
          \label{eq:a_la}
	  a_\lambda(n) &= \prod_{1 \leq i < j \leq
            k}\frac{\rho_i-\rho_j}{2n+\rho_i+\rho_j}\prod_{i=1}^k\frac{(2n-2i)!}{2(\rho_i+k)!(\rho_i
                         +2n-k-1)!(\rho_i+n)}.
        \end{align}
         If $n > \ell(\lambda)$, we have
	 \begin{align}
          \label{eq:a_la'}
          a_\lambda(n) &= 2\prod_{1 \leq i < j \leq
                         n}\frac{\rho_i-\rho_j}{2n+\rho_i+\rho_j}\prod_{i=1}^n\frac{(2n-2i)!}{2(\rho_i+n)!^2}.
                                 \end{align}        
\end{lemma}

\begin{proof}
Let $n>k\geq \ell(\lambda)$ and note that $\rho_i=-i$ if $i>k$. We split the product in \eqref{eq:dimWeyl} as follows,
\begin{equation*}
so_\lambda(1^{2n}) = N_\lambda(k) M_\lambda(n,k) \frac{D(k+1,n)}{D(1,n)}
\end{equation*}
with 
\begin{equation*}
N_\lambda(k) = \prod_{1\leq i<j\leq k} (\rho_i-\rho_j)(\rho_i+\rho_j+2n),\qquad
M_\lambda(n,k) = \prod_{i=1}^k \prod_{j=k+1}^n (\rho_i+j)(\rho_i-j+2n)
\end{equation*}
and for all $1\leq p<n$
\begin{equation*}
D(p,n) = \prod_{p\leq i<j\leq n} (-i+j)(-i-j+2n).
\end{equation*}
The products which do not involve $\rho_i$ and $\rho_j$ can be performed in terms of factorials, yielding
\begin{equation*}
M_\lambda(n,k) = \prod_{i=1}^k (\rho_i+n) \frac{(\rho_i+2n-k-1)!}{(\rho_i+k)!},\qquad
D(p,n) = \prod_{i=p}^{n-1} \frac{(2n-2i)!}{2}.
\end{equation*}
This gives the following expression for $so_\lambda(1^{2n})$,
\begin{equation*}
so_\lambda(1^{2n}) = \prod_{1\leq i<j\leq k} (\rho_i-\rho_j)(\rho_i+\rho_j+2n) \prod_{i=1}^k 2(\rho_i+n) \frac{(\rho_i+2n-k-1)!}{(\rho_i+k)!(2n-2i)!}
\end{equation*}
Notice that we derived this formula assuming $k<n$. In the case
$\ell(\lambda)=n$ and plugging $k=n$ in the formula above we obtain
$2\cdot so_\lambda(1^{2n})$
since
\[ \frac{1}{D(1,n)}=\prod_{1 \leq
   i  \leq n-1}\frac{2}{(2n-2i)!} = \frac{1}{2}\prod_{1 \leq i  \leq
   n}\frac{2(\rho_i+n)(2n-n-1+\rho_i)!}{(\rho_i+n)! (2n-2i)!}\]
for any partition $\lambda$ of length $n$ and it yields the RHS of
\eqref{eq:dimWeyl} multiplied by two.
Therefore by \eqref{eq:dimOrthogonal} we have that
\begin{equation*}
o_\lambda(1^{2n}) = \prod_{1\leq i<j\leq k} (\rho_i-\rho_j)(\rho_i+\rho_j+2n) \prod_{i=1}^k 2(\rho_i+n) \frac{(\rho_i+2n-k-1)!}{(\rho_i+k)!(2n-2i)!}
\end{equation*}
for $n > k$, and $n
        = k = \ell(\lambda)$.
Plugging the hook-product formula (e.g. \cite{Stanley:EC2})
      \[ \hook_\lambda = \prod_{1 \leq i < j
          \leq k}\frac{1}{\rho_i-\rho_j} \prod_{i=1}^k(\rho_i+k)!\]
for all $k\geq \ell(\lambda)$, we obtain~\eqref{eq:a_la}.

Notice now that substituting $k=n-1$ in~\eqref{eq:a_la} we obtain
        \begin{align*}
	  a_\lambda(n) &=\prod_{1 \leq i < j \leq
            n-1}\frac{\rho_i-\rho_j}{2n+\rho_i+\rho_j}\prod_{i=1}^{n-1}\frac{(2n-2i)!}{2(\rho_i
                         +n)!^2}.
        \end{align*}
 It has the same form as \eqref{eq:a_la'}, but the bound on the products is $n-1$ instead of $n$, and there is a missing factor 2 in front. It is easy to see that the bound can be extended to $n$ by compensating by a factor 2. Indeed, since $\rho_n=-n$, one has $\prod_{i=1}^{n-1} \frac{\rho_i-\rho_n}{2n+\rho_i+\rho_n} = 1$, and $\prod_{i=1}^n\frac{(2n-2i)!}{2(\rho_i+n)!^2} = \frac{1}{2} \prod_{i=1}^{n-1}\frac{(2n-2i)!}{2(\rho_i+n)!^2}$.
	\end{proof}

	Using \cref{eq:a_la}, we can (and we do) promote the definition of $a_\lambda(n)$ to  nonnegative integer $k$-tuples $\la \in \Z_{\geq 0}^k$. This is an antisymmetric function of the $k$ parameters $\rho_i=\lambda_i-i$.
In what follows it will also be natural to think of the Schur function $s_\lambda$ as a function of these parameters. For an integer partition $\lambda$ with $\ell(\lambda) \leq k$, we will thus write 
$$s_\lambda = s^{(\rho_1, \dots, \rho_k)}.$$
 We extend this definition antisymmetrically to $k$-tuples $\rho \in \Z^k$. In particular $s^{(\rho_1, \dots, \rho_k)}$ vanishes if two of the $\rho_i$ are equal.

 The following lemma describes the action of Virasoro constraints on scaled Schur functions. Its proof will be more natural after the Boson-Fermion correspondence is introduced, so we postpone the proof to Section~\ref{sec:formalNSolutions}, Lemma~\ref{lemma:virasoroBF}.
\begin{lemma}[Action of the Virasoro operator on scaled Schur functions]\label{lemma:virasoroDelayed}
	Let $r\geq 1$, $u^{-1} = 2n$ and recall the notation 
$L^{Z}_r$ from~\eqref{eq:defLZi}. For any vector $\rho \in
\mathbb{Z}^k$ with $\rho_i\geq -i$ and $\rho_{k-i} = i-k$ for $0 \leq i < r$ (in
particular $k>r)$ one has
\begin{align}
  \label{eq:VirasoroOnSchur}
	\left(L^{Z}_r\Big|_{b=1} \right)
	      s^{(\rho_1, \dots, \rho_k)}(\mathbf{p/2})
	     =
	     -\frac{t}{2}\delta_{r,1} s^{(\rho_1, \dots, \rho_k)}(\mathbf{p/2})
	     +
	     \sum_{i=1}^{k-r} \Bigl(n+\rho_i\Bigr) s^{(\rho_1, \dotsc, \rho_k)-r \epsilon_i}(\mathbf{p/2}).
\end{align}
where $\epsilon_i=(0,\dots,1,\dots 0)$ is the vector with a unique $1$
in $i$-th position.
\end{lemma}

We can now prove the theorem.
    
 \begin{proof}[Proof of Theorem~\ref{thm:schurExpansion}]
	 We will show that the RHS of~\eqref{eq:schurExpansion}
         satisfies the same Virasoro constraints as
         $\tau_{b=1}^{Z}(t;\pp,u)$, i.e.~it is annihilated by the
         Virasoro operator $\left(L^{Z}_r\Big|_{b=1} \right)$ for each $r\geq 1$. Multiplying by $p_r$ and summing over $r$, this will show that it satisfies the evolution equation~\eqref{eq:evolution} with $b=1$, which is enough to conclude since this equation has a unique solution in $\mathbb{Q}[\pp][[u,t]]$. 

	 Fix $r\geq 1$, and make the change of variables
         $u^{-1}=2n$. We will prove that for each partition $\lambda$
         the coefficients of $s_\lambda(\mathbf{p/2})$ in the Schur-expansion of
$$
	 \left(L^{Z}_r\Big|_{b=1} \right) \sum_{\mu} t^{|\mu|}a_\mu(n) s_{\mu}(\mathbf{p/2})$$
         is equal to zero. 
	
	 We first claim that it is enough to show
         that for each positive integer $k$ and for each $\rho \in
\mathbb{Z}^k$ with $\rho_i\geq -i$ and $\rho_{k-i} = i-k$ for $0 \leq i < r$ we have, in previous notation
	 \begin{equation}\label{eq:toprove}
		 \sum_{i=1}^{k}(n+\rho_i+r)\frac{a_{\lambda+\epsilon_ir}(n)}{a_\lambda(n)} = \delta_{r,1}\frac{1}{2},
               \end{equation}
               as an identity between rational functions of the formal
               variable $n$.
	 Indeed, fix a partition $\lambda$ and let $k\geq\ell(\lambda)+r$. From the previous lemma, the coefficient of the Schur function $s_\lambda(\mathbf{p/2})$ in the quantity
$$
	 \left(L^{Z}_r\Big|_{b=1} \right) \sum_{\mu} t^{|\mu|}a_\mu(n) s_{\mu}(\mathbf{p/2})$$
	 is equal to $t^{|\lambda|}$ times 
 $$ -\frac{t}{2}\delta_{r,1} a_\lambda(n)
	 + t^r\sum_{i=1}^{k-r}\sum_{\nu}  \sgn(\sigma^{\nu,\lambda})\Bigl(n+\nu_i-i\Bigr)         a_\nu(n)$$
	 where the second sum is taken over all partitions $\nu$ such that $\nu - r\epsilon_i\equiv\lambda$ (here we write $\lambda\equiv\mu$ if there exists a permutation $\sigma$ such that $\lambda_j-j=\mu_{\sigma(j)}-\sigma(j)$ for all $j$, and if this is the case we write $\sigma=\sigma^{\lambda,\mu}$).
		 Now, note that there exists $i$ such that $\nu -
                 r\epsilon_i\equiv\lambda$ if and only if
	 there exists $j$
                 such that $\lambda+r \epsilon_j\equiv\nu$, and
                 if this is the case the permutations $\sigma^{\nu -
                   r\epsilon_i,\lambda}$ and $\sigma^{\lambda+r
                   \epsilon_j,\nu}$ are inverse to each other, thus
                 have the same sign. Therefore, using that the symbol
                 $a_\nu$ is antisymmetric in the $\nu_i-i$, the last
                 quantity can be rewritten as
 $$ -\frac{t}{2}\delta_{r,1} a_\lambda(n)
	 + t^r\sum_{j=1}^{k}\Bigl(n+\lambda_j-j+r\Bigr)
	 a_{\lambda+r\epsilon_j}(n),$$
	 which, admitting~\eqref{eq:toprove}, is equal to zero.

	 \medskip

	 We will now prove the identity~\eqref{eq:toprove}, between rational functions of $n$.
	Note that for a fixed partition $\lambda$ the LHS of~\eqref{eq:toprove} does not depend on the choice
	of $k$ assuming that $n>k \geq \ell(\lambda)+r$ and that $\rho_i
        := \lambda_i-i$ (with the convention that $\rho_i=-i$ for $i >
        \ell(\lambda)$). This is a
        consequence of \eqref{eq:a_la} which implies that
        $a_{\lambda+\epsilon_ir}(n)=0$ for $i > \ell(\lambda)+r$
        (since $\rho_i+r = \rho_{i-r}$).

	 Fix a positive integer $k \geq \ell(\lambda)+r$. Using \eqref{eq:a_la} we can rewrite the LHS of~\eqref{eq:toprove}, for $n$ large enough, as
        \begin{equation}
          \label{eq:analytic}
          \sum_{i=1}^{k} \frac{(n+\rho_i)}{(\rho_i+k+r)_{(r)}(\rho_i+2n-k+r-1)_{(r)}} \prod_{j\neq i} \frac{(\rho_i-\rho_j+r)}{(\rho_i-\rho_j)}\frac{(\rho_i+\rho_j+2n)}{(\rho_i+\rho_j+2n+r)},
        \end{equation}
        where $(n)_{(m)}:=\prod_{i=1}^m(n-i+1)$ denotes the falling
        factorial. 
	 Note that in 	\eqref{eq:analytic} we have $\rho_{k-i}=i-k$ for $0 \leq i < r$.
   	We will thus treat  \eqref{eq:analytic} 
 as a rational function $f^{(k)}_{n,r} : \overline{\C}^{k-r} \to \overline{\C}$
	 of $(\rho_1,\dots,\rho_{k-r}) \in \overline{\C}^{k-r}$, with $\overline{\C} = \C\cup
        \{\infty\}$ .

\smallskip

	 We will now show by induction on $k\geq r+1$ that for  $n$ large enough, the function $f^{(k)}_{n,r}$ is constant as a function of $(\rho_1,\dots,\rho_{k-r})$, and that moreover it is equal to $\frac{1}{2}\delta_{r,1}$.

	 Let $k\geq r+1$. We will study the behaviour of $f^{(k)}_{n,r}$ at all possible poles (including infinity), which we now enumerate.
        \begin{itemize}[itemsep=0pt, topsep=0pt,parsep=0pt, leftmargin=12pt]
        \item $\rho_i=\rho_j = z$ for some $i\neq j, z \in \overline{\C}$. The residue $\Res_{\rho_i=\rho_j=z}f^{(k)}_{n,r}(\rho)$ is equal to
          \begin{align*}
         \frac{(n\splus{}z)}{(z\splus{}k\splus{}r)_{(r)}(z\splus{}2n\sminus{}k\splus{}r\sminus{}1)_{(r)}}
            \frac{2z\splus{}2n}{2z\splus{}2n\splus{}r} \times
            \prod_{h\neq {i,j}} \frac{(z\sminus{}\rho_h\splus{}r)}{(z\sminus{}\rho_h)}\frac{(z\splus{}\rho_h\splus{}2n)}{(z\splus{}\rho_h\splus{}2n\splus{}r)}(r\sminus{}r)
	    &= 0.
          \end{align*}
                  \item $\rho_i=\sminus{}\rho_j\sminus{}2n\sminus{}r = z\sminus{}n\sminus{}\tfrac{r}{2}$ for some $i\neq j, z \in \overline{\C}$. The residue $\Res_{\rho_i=\sminus{}\rho_j\sminus{}2n\sminus{}r= z\sminus{}n\sminus{}\tfrac{r}{2}}f^{(k)}_{n,r}(\rho)$ is equal to
			  \begin{align*} \sum_{\pm}\frac{\sminus{}r (\spm{}z\sminus{}\tfrac{r}{2})}{(\spm{}z\splus{}k\splus{}\tfrac{r}{2}\sminus{}n)_{(r)}(\spm{}z\splus{}n\sminus{}k\splus{}\tfrac{r}{2}\sminus{}1)_{(r)}}\frac{\spm{}2z\splus{}r}{\spm{}2z}
            \prod_{h\neq {i,j}}
            \frac{(\spm{}z\sminus{}\rho_h\splus{}\tfrac{r}{2}\sminus{}n)}{(\spm{}z\sminus{}\rho_h\sminus{}\tfrac{r}{2}\sminus{}n)}
            \frac{(\spm{}z\splus{}\rho_h\splus{}n\sminus{}\tfrac{r}{2})}{(\spm{}z\splus{}\rho_h\splus{}n\splus{}\tfrac{r}{2})}\\
            =	    	    \frac{(z\sminus{}\tfrac{r}{2})(z\splus{}\tfrac{r}{2})(\sminus{}r\splus{}(\sminus{}1)^{2r}r)}{z(\sminus{}z\splus{}k\splus{}\tfrac{r}{2}\sminus{}n)_{(r)}(\sminus{}z\splus{}n\sminus{}k\splus{}\tfrac{r}{2}\sminus{}1)_{(r)}}
	    \prod_{h\neq {i,j}}     \frac{(\sminus{}z\sminus{}\rho_h\splus{}\tfrac{r}{2}\sminus{}n)}{(\sminus{}z\sminus{}\rho_h\sminus{}\tfrac{r}{2}\sminus{}n)}\frac{(\sminus{}z\splus{}\rho_h\splus{}n\sminus{}\tfrac{r}{2})}{(\sminus{}z\splus{}\rho_h\splus{}n\splus{}\tfrac{r}{2})}
	    =0.          
          \end{align*}
          \item $\rho_i= j\sminus{}1\sminus{}k\sminus{}r$ for some $i \in [k-r], j \in [r]$. In
            order to show that there is no pole it is enough to recall
            that the function $f^{(k)}_{n,r}: \overline{\C}^{k-r} \to \C$ is
            equal to the function
                    \begin{equation*}
          \sum_{i=1}^{k\splus{}N} \frac{(n\splus{}\rho_i)}{(\rho_i\splus{}k\splus{}N\splus{}r)_{(r)}(\rho_i\splus{}2n\sminus{}k\sminus{}N\splus{}r\sminus{}1)_{(r)}} \prod_{j\neq i} \frac{(\rho_i\sminus{}\rho_j\splus{}r)}{(\rho_i\sminus{}\rho_j)}\frac{(\rho_i\splus{}\rho_j\splus{}2n)}{(\rho_i\splus{}\rho_j\splus{}2n\splus{}r)},
        \end{equation*}
        for any positive integer $N$, where $\rho_i := -i$ for $k  < i
        \leq k+N$ (assuming $n > k+N$). Indeed, this is true for
        arbitrary $(\rho_1,\dots,\rho_{k-r}) \in \Z_{\geq}^{k-r}$, and since
        $f^{(k)}_{n,r}$ is a rational function, it also holds true for
        arbitrary $(\rho_1,\dots,\rho_{k-r}) \in \overline{\C}^{k-r}$. Taking
        $N \geq j$ we can see that in fact there is no pole in
        $\rho_i= j\sminus{}1\sminus{}k\sminus{}r$.
        \item the same argument also shows that there is no
          pole in $\rho_i= j\sminus{}2n\splus{}k\sminus{}r$ for some $i \in [k-r], j \in [r]$.
          \item $\rho_i=\infty$. Note that the limit $\rho_i \to \infty$ is equal
            to
			\begin{align}\label{eq:quantity}
            \sum_{j\neq i}^{k} \frac{(n\splus{}\rho_j)}{(\rho_j\splus{}k\splus{}r)_{(r)}(\rho_j\splus{}2n\sminus{}k\splus{}r\sminus{}1)_{(r)}} \prod_{h\neq i,j} \frac{(\rho_j\sminus{}\rho_h\splus{}r)}{(\rho_j\sminus{}\rho_h)}\frac{(\rho_j\splus{}\rho_h\splus{}2n)}{(\rho_j\splus{}\rho_h\splus{}2n\splus{}r)}.
          \end{align}
	If $k=r+1$, \eqref{eq:quantity} is a finite quantity (which doesn't depend on any indeterminate).
If $k>r+1$, \eqref{eq:quantity} is equal to $f^{(k-1)}_{n-1,r}(\tilde{\rho})$, where $\tilde{\rho}=(\rho_1+1,\dots,\widehat{\rho_i+1}, \dots, \rho_{k-r}+1)$ and the hat means that the $i$-th value is ommitted.
By the induction hypothesis, if $n$ is large enough this quantity is independent of the $\rho_i$ and it is equal to $\frac{1}{2}\delta_{r,1}$.
        \end{itemize}

Since we examined all possible poles, we have just shown that the rational function $f^{(k)}_{n,r}$ is bounded, therefore by Liouville's theorem it is a constant, as a function of $(\rho_1,\dots,\rho_{k-r})$. 
Moreover, when $k>r+1$, this constant can be evaluated from the last bullet point (where we used the induction hypothesis), which  shows that it is equal $\frac{1}{2}\delta_{r,1}$ for $n$ large enough. 

\smallskip

Therefore to conclude the induction it only remains to adress the base case $k=r+1$. In this case we already know that $f^{(k)}_{n,r}$ is constant, and we use \eqref{eq:quantity} above to evaluate it (for convenience in the notation we choose to evaluate  $f^{(k)}_{n+1,r}$ instead)

\begin{align*}
		f^{(k)}_{n+1,r}(\rho_1) = f^{(k)}_{n+1,r}(\infty) 
		&=\sum_{i=1}^{r}
          \frac{(n\sminus{}i)}{(2r\sminus{}i)_{(r)}(2n\sminus{}1\sminus{}i)_{(r)}} \prod_{j\neq
            i} \frac{(\sminus{}i\splus{}j\splus{}r)}{(\sminus{}i\splus{}j)}\frac{(\sminus{}i\sminus{}j\splus{}2n)}{(\sminus{}i\sminus{}j\splus{}2n\splus{}r)}
          \\
          &= \sum_{i=1}^{r}
          \frac{(n\sminus{}i)(r\sminus{}i)!(2n\sminus{}1\sminus{}r\sminus{}i)!}{(2r\sminus{}i)!(2n\sminus{}1\sminus{}i)!}
            \frac{(\sminus{}1)^{i\sminus{}1} (2r\sminus{}i)!}{r(i\sminus{}1)!(r\sminus{}i)!(r\sminus{}i)!}\cdot\\
	    &\hspace{4cm}\frac{(2n\sminus{}1\sminus{}i)!(2n\sminus{}1\sminus{}i)!(2n\sminus{}2i\splus{}r)}{(2n\sminus{}1\sminus{}r\sminus{}i)!(2n\splus{}r\sminus{}1\sminus{}i)!2(n\sminus{}i)}\\
          &=\frac{1}{2}\sum_{i=1}^{r}(\sminus{}1)^{i\sminus{}1}\frac{(2n\sminus{}i\sminus{}1)!(2n\sminus{}2i\splus{}r)}{r(i\sminus{}1)!(r\sminus{}i)!(2n\splus{}r\sminus{}i\sminus{}1)!}\\
            &= \frac{1}{2r!}\sum_{i=0}^{r-1}(\sminus{}1)^i\binom{r\sminus{}1}{i}\frac{2n\sminus{}2\splus{}r\sminus{}2i}{(2n\sminus{}2\splus{}r\sminus{}i)_{(r)}}.
        \end{align*}
        In order to compute the last expression, we consider it as a 
	rational function of $2n$, $f_{r}: \overline{\C} \to \overline{\C}$
        and we will apply Liouville's theorem again. The only possible
        poles of $f_{r}$ (all of them are at most simple) are in $2n = 1\splus{}(\sminus{}1)^\epsilon
        j$, where $0 \leq j \leq r\sminus{}1$ and $\epsilon \in \{0,1\}$.
	But for $r>1$, the residue $\Res_{2n = 1\sminus{}  j}f_{r}(x)$ is null 
        because the terms of the sum corresponding to the indices $i$ and
	$r\sminus{}1\sminus{}j\sminus{}i$ cancel out, and the same is true for $\Res_{2n = 1\splus{}
	j}f_{r}(x)$. We thus obtain 
                \[ \Res_{x = \infty} f_r(x) =
                  \frac{1}{2}\delta_{r,1},\]
		which implies that $f_{n+1,r}$  (hence $f_{n,r}$) is equal to 
                $\frac{1}{2}\delta_{r,1}$.
This concludes the inductive part of the proof.

\medskip

We have thus proved \eqref{eq:toprove} for infinitely many values of $n$, therefore it is true as an identity between rational functions, and the proof is complete.
\end{proof}

\subsection{The conjecture of Oliveira and Novaes}
\label{subsec:OliveiraNovaes}

       An immediate corollary of Theorem~\ref{thm:schurExpansion} is the
        proof of \cite[Conjecture 1]{OliveiraNovaes2021}. Before we
        state it, let us introduce the notation from
        \cite{OliveiraNovaes2021} so that we can translate this
        conjecture into our framework. Let $N$ be a formal parameter
        and $m\geq 1$ be an integer, and define
        \begin{equation*}
          [N]_\lambda^{(2)} := Z_\lambda(1^N), \ \ \ 
          \{N\}_\lambda := \frac{m!}{\chi_\lambda(1^m)}o_\lambda(1^N),
\end{equation*}
where we recall for reference that $Z_\lambda(1^N) = \prod_{\square\in\lambda} (N+c_2(\square))$ and $\chi_\lambda(1^m)/m! = 1/\hook_\la$ which is the hook length formula. Let furthermore
\begin{equation*}
          G_{\lambda,\gamma} := \sum_{\mu \vdash
                                m}\frac{m!}{z_\mu}\omega_\lambda(\mu)\chi_\gamma(\mu),
        \end{equation*}
        where $\chi_\gamma$ is the character of the irreducible 
        representation of the symmetric group $\Sym{m}$, and
        $\omega_\lambda$ is the zonal spherical function of the
        Gelfand pair $(\Sym{2m},H_m)$ with $H_m$ the
        hyperoctahedral group, indexed respectively by partitions $\gamma$ and $\lambda$ of $m$. Irreducible characters of the symmetric
        group and zonal spherical function of the
        Gelfand pair $(\Sym{2m},H_m)$ are related to Schur $s_\lambda$
        and zonal $Z_\lambda$
        symmetric functions, respectively, by their expansions in
        the power-sum basis
        \begin{align*}
         \frac{\chi_\lambda(1^m)}{m!}J^{(1)}_\lambda &=
          s_\lambda = \sum_{\mu \vdash m} \chi_\lambda(\mu)
                                                      \frac{p_\mu}{z_\mu},\\
          J^{(2)}_\lambda &= Z_\lambda =   2^m\sum_{\mu \vdash m} \omega_\lambda(\mu)
                                                      \frac{p_\mu}{2^{\ell(\mu)}z_\mu}.
          \end{align*}
We refer to ~\cite{Macdonald1995} for background. We have
\begin{theorem}[{\cite[Conjecture 1]{OliveiraNovaes2021}}]\label{thm:OliveiraNovaes}
          The following identity holds:
          \[ \sum_{\lambda \vdash m}
            \frac{\chi_{2\lambda}(1^{2m})}{[N]_\lambda^{(2)}}G_{\lambda,\gamma}
              = \frac{(2m)!}{2^mm!}\frac{\chi_\gamma(1^m)}{\{N\}_\gamma}.\]
          \end{theorem}

          \begin{proof}
\cref{thm:schurExpansion} is an explicit equality between an expansion on zonal symmetric functions and an expansion on scaled Schur functions. Using the notations from \cite{OliveiraNovaes2021}, the hook length formula and specializing $t=1$, it reads
\begin{equation} \label{SchurExpansionON}
\sum_{m\geq 0} \sum_{\lambda\vdash m} \frac{\chi_{2\lambda}(1^{2m})}{(2m)!} \frac{Z_\lambda(\pp)}{[N]^{(2)}_\lambda} = \sum_{m\geq 0} \sum_{\lambda\vdash m} \frac{\chi_\lambda(1^m)}{m!} \frac{s_\lambda(\mathbf{p/2})}{\{N\}_\lambda}.
\end{equation}

Consider the standard Hall scalar product $\langle,\rangle=\langle,\rangle_1$ on the space of 
symmetric functions with respect to the variables $\pp':=\mathbf{p/2}$. By
definition (see e.g.~\cite{Macdonald1995}) we have that
\begin{align*}
  \langle s_\lambda(\pp'), s_\mu(\pp')\rangle = \delta_{\la,\mu} \ \ , \ \ 
  \langle p_\lambda, p_\mu \rangle = 2^{2\ell(\lambda)}z_\la\delta_{\la,\mu}.
\end{align*}
We will now take the scalar product of both sides of \eqref{SchurExpansionON} with $s_\gamma(\pp')$. The RHS is
equal to
\begin{equation*}
\left\langle\sum_\lambda \frac{\chi_\lambda(1^m)}{m!} \frac{s_\lambda(\pp')}{\{N\}_\lambda}, s_\gamma(\pp')\right\rangle =   \frac{1}{m!}\frac{\chi_\gamma(1^m)}{\{N\}_\gamma}.
\end{equation*}
To evaluate the LHS, we expand the summand in the power-sum basis
\[ \frac{\chi_{2\lambda}(1^{2m})}{(2m)!} \frac{Z_\lambda(\pp)}{[N]^{(2)}_\lambda}  =
\frac{2^m \chi_{2\lambda}(1^{2m}) }{(2m)![N]^{(2)}_\la}  \sum_{\mu \vdash
                                m}\frac{\omega_\lambda(\mu)2^{-\ell(\mu)}p_\mu}{z_\mu},\]
which directly leads to
\[ \left\langle
\frac{\chi_{2\lambda}(1^{2m})}{(2m)!} \frac{Z_\lambda(\pp)}{[N]^{(2)}_\lambda},
  s_\gamma(\pp') \right\rangle = \frac{2^m \chi_{2\lambda}(1^{2m}) }{(2m)![N]^{(2)}_\la}  \sum_{\mu \vdash
                                m}\frac{\omega_\lambda(\mu)\chi_\gamma(\mu)}{z_\mu
                                } = \frac{2^m
                                  \chi_{2\lambda}(1^{2m})
                                }{(2m)![N]^{(2)}_\la}G_{\lambda,\gamma}.\]
                              Therefore we have
                              \[ \sum_{\lambda \vdash m}\frac{2^m
                                  \chi_{2\lambda}(1^{2m})
                                }{(2m)![N]^{(2)}_\la}G_{\lambda,\gamma}
                                  =
                                  \frac{1}{m!}\frac{\chi_\gamma(1^m)}{\{N\}_\gamma}\]
                                and multiplying both sides by $\frac{(2m)!}{2^m}$
gives the desired equality.
\end{proof}

\subsection{On the symplectic case $b=-1/2$.}
To conclude this section, we mention that in the same way as the value $b=1$ is related to the orthogonal group, the case $b=-1/2$ is related to the symplectic group. We will not state symplectic analogues of all results of this section since the two values are related by a form of duality. For example, the following symplectic analogue of Theorem~\ref{thm:schurExpansion} is in fact a direct corollary of it:
      \begin{corollary}[Explicit expansion of $\tau_{b=-1/2}^{Z}$ in Schur functions]\label{cor:schurExpansionSymplectic}
	The function  $\tau_{b=-1/2}^{Z}(t;\pp,u)$ defined by its expansion~\eqref{eq:defmain}, has the following expansion in Schur functions of the variables~$\pp$,
	\begin{align}\label{eq:schurExpansionSymplectic}
		\tau_{b=-1/2}^{Z}(t;\pp,u) =\sum_{n\geq 0} (4t)^n \sum_{\lambda}\frac{s_\lambda(\pp)}{\hook_\lambda^2\cdot
		o_{\lambda^t}(1^{2u^{-1}})} = \sum_{n\geq 0} (-4t)^n \sum_{\lambda}\frac{s_\lambda(\pp)}{\hook_\lambda^2\cdot
	sp_{\lambda}(1^{-2u^{-1}})},
	\end{align}
	where $sp_{\lambda}(1^{2n})$ denotes the dimension of the
        irreducible representation of the heighest weight $\lambda$ of
        the symplectic orthogonal group $Sp(2n)$.
        This identity holds in $\mathbb{Q}(u)[\pp][[t]]$ and in $\mathbb{Q}[\pp][[u,t]]$.
      \end{corollary}

      \begin{proof}
       This identity is obtained by applying the 
       transformation $\omega_2(p_r) := 2(-1)^{r-1}p_r$ on both sides
       of \eqref{eq:schurExpansion} and using the following classical
       identities (see~\cite{ElSamraKing1979,Stanley1989,Macdonald1995}):
       \begin{align*}
         \omega_2(s_\lambda(\mathbf{p/2})) &= s_{\lambda^{t}}(\pp),\\
         \omega_2(J_\lambda^{(1)}(\pp))
                                           &=2^{|\lambda|}J_{\lambda^t}^{(-1/2)}(\pp),\\
         \prod_{\square \in \lambda}\frac{1}{u^{-1}+c_{(1)}(\square)}
                                           &= \prod_{\square \in
                                             \lambda^t}\frac{2^{-1}}{u^{-1}/2+c_{(-1/2)}(\square)},\\
         j_{\lambda}^{(1)} &=  4^nj_{\lambda^t}^{(-1/2)},\\
         o_{\lambda^t}(1^u) &=  (-1)^{|\lambda|}sp_{\lambda}(1^{-u}).\qedhere
         \end{align*}
        \end{proof}
 Using the same techniques applied to \eqref{eq:schurExpansionSymplectic} we can obtain an analogous
  result to \cref{thm:OliveiraNovaes} involving symplectic characters
  and symplectic zonal spherical functions related to symplectic
  zonal polynomials $J_\lambda^{(-1/2)}$ in the same way as zonal
  spherical functions $\omega_\lambda$ are related to zonal $J_\lambda^{(1)}$.
  polynomials. Details are left to the reader.

\section{Pfaffians and formal large BKP hierarchy}
\label{sec:BKP}

\subsection{Pfaffians} We will use the following definition for the Pfaffian.
\begin{definition}\label{def:Pfaffian}
	The \emph{Pfaffian} of a skew-symmetric matrix $A$ of even size $n$ is the quantity $\Pf(A)$ defined by:
       \begin{equation}
         \label{eq:DefPf}
	       \Pf(A) := \frac{1}{2^{n/2} ( n/2)!}\sum_{\sigma \in \Sym{n}}\sgn(\sigma)\prod_{i=1}^{ n/2}a_{\sigma(2i-1),\sigma(2i)}.
       \end{equation}
\end{definition}
We have the following classical Schur's Pfaffian identity.
				\begin{lemma}
                \label{lem:SchurPfaffian}
                Let $n$ be an integer and $x_1, \dotsc, x_n$ be $n$
                real variables. Set $x_{n+1}:=0$ by convention, and
                $x_{ij} := \frac{x_i-x_j}{x_i+x_j}$ for $i,j=1,
                \dotsc, n+1$, with $x_{ij}=\begin{cases} 1 \text{ if }
                  x_i=x_j=0 \text{ and } i\neq j,\\ 0 \text{ if }
                  x_i=x_j=0 \text{ and } i=j.\end{cases}$. Then
                \begin{equation}
                  \label{eq:SchurPfaffian}
                  \prod_{1 \leq i < j \leq n}\frac{x_i-x_j}{x_i+x_j} = \begin{cases} \Pf( x_{ij})_{1\leq i,j\leq n} \qquad & \text{for $n$ even,}\\
				  \Pf( x_{ij})_{1\leq i,j\leq n+1} & \text{for $n$ odd.} \end{cases}
                  \end{equation}
                \end{lemma}

Notice that the case $n$ odd is a consequence of the case $n$ even with our conventions, since then $\prod_{1 \leq i < j \leq n}\frac{x_i-x_j}{x_i+x_j} = \prod_{1 \leq i < j \leq n+1}\frac{x_i-x_j}{x_i+x_j}$.

We will also need the following minor summation formula for Pfaffians.

	    \begin{proposition}[\cite{IshikawaWakayama1995}]
        \label{prop:CBForPfaffians}
        Suppose that $n$ is a positive even integer, $\N\cup
        \{\infty\} \ni N \geq n$ is an integer, $B$ is a $n \times
        N$ matrix and $A$ is a $N\times N$ skew-symmetric matrix. Then
        \begin{align}
          \label{eq:CBForPfaffians}
\sum_{\substack{I \subset [N],\\ |I|=n}}\det(B_{i,j})_{i \in [n], j \in I}\Pf(A_{i,j})_{i,j \in I} = \Pf(BAB^t).
                \end{align}
\end{proposition}
Proposition~\ref{prop:CBForPfaffians} will be important to handle matrix integrals in Section~\ref{sec:BGW}. For the rest of Section~\ref{sec:BKP} we will only need an elementary special case, namely the fact that for matrices $A$ and $B$ of the same size, one has 
\begin{align}\label{eq:simpleCBForPfaffians}
	\det(B)\Pf(A) = \Pf(BAB^t).
\end{align}

The following theorem shows that the inverse dimension $1/o_\lambda(1^{2n})$ has a Pfaffian structure. We write it in terms of the coefficient $a_\lambda(n)$ previously introduced.
\begin{theorem}[Dimension of Orthogonal representations and Pfaffians]
        \label{theo:Pfaffian}
        Let $n$ be an integer, and suppose that $\ell(\la) \leq n$.
	Then we have
\begin{align}\label{eq:alambdaPfaffian}
a_\lambda(n) &= \frac{1}{\hook_\lambda^2\cdot o_\lambda(1^{2n})} = \begin{cases} \displaystyle\prod_{k=1}^{n-1}(2k)! \cdot \Pf(a_{\lambda_i+n-i,\lambda_j+n-j})_{1 \leq i,j \leq n}\ \text{ for $n$ even,}\\
\displaystyle\prod_{k=1}^{n-1}(2k)! \cdot \Pf(a_{\lambda_i+n-i,\lambda_j+n-j})_{1 \leq i,j \leq n+1}\ \text{ for $n$ odd,} \end{cases}\\
\mbox{where }\ \ \ a_{i,j} &= \begin{cases} \frac{i-j}{4(i+j)\ i!^2 j!^2} &\mbox{ for }\  i,j\geq 1,\\
\frac{1}{2i!^2} &\mbox{ for }\  j\in \{-1,0\}, i > 0,\\
1 &\mbox{ for }\ j=-1, i=0,\\
\frac{-1}{2j!^2} &\mbox{ for }\ i\in \{-1,0\}, j > 0,\\
-1 &\mbox{ for }\ i=-1, j=0,\\
{\scriptstyle{0}} &\mbox{ for }\  i,j \in \{-1,0\},i=j,\end{cases} \nonumber
\end{align}
and $\lambda_{n+1}=0$ by convention.
\end{theorem}

    \begin{proof}
There are four cases to analyze, depending on whether $\ell(\lambda)<n$ or $\ell(\lambda)=n$ and whether $n$ is even or odd. In this proof we denote 
\begin{equation*}
x_i=\rho_i+n = \lambda_i+n-i\qquad \forall i=1,\dotsc, n,
\end{equation*}
and $x_{n+1}:=0$ as in \cref{lem:SchurPfaffian}.

Let $V_{ij} = \delta_{ij}/(2x_i!^2)$ for $1\leq i,j\leq n$ and $V_{n+1, j} = V_{j, n+1} = \delta_{j,n+1}$. In particular
\begin{equation*}
\det (V_{ij})_{1\leq i,j\leq n} = \prod_{i=1}^n \frac{1}{2\ x_i!^2},\qquad \text{and}\qquad \det (V_{ij})_{1\leq i,j\leq n+1} = \det (V_{ij})_{1\leq i,j\leq n}.
\end{equation*}

\begin{itemize}[itemsep=0pt, topsep=0pt,parsep=0pt, leftmargin=12pt]
\item Assume first that $\ell(\la)=n$ and $n$ is even. Using \eqref{eq:a_la} and applying the Schur-Pfaffian identity \eqref{eq:SchurPfaffian} with $x_i=\rho_i+n = \lambda_i+n-i$,  we obtain
      \begin{equation*}
        a_\la(n) = \prod_{k=1}^{n-1} (2k)! \det(V_{ij})_{1\leq i,j\leq n}\ \Pf\left(\frac{x_i-x_j}{x_i+x_j}\right)_{1 \leq i,j \leq n}.
      \end{equation*}
		We conlude by using the identity~\eqref{eq:simpleCBForPfaffians} for $B=(V_{ij})_{1
  \leq i,j \leq n}$ and $A=\left(\frac{x_i-x_j}{x_i+x_j}\right)_{1
  \leq i,j \leq n}$. It gives
\begin{align*}
a_\la(n) &= \prod_{k=1}^{n-1} (2k)!\ \Pf\left(\frac{x_i-x_j}{x_i+x_j} \frac{1}{4\,x_i!^2 x_j!^2}\right)_{1 \leq i,j \leq n}\\
&= \prod_{k=1}^{n-1} (2k)!\ \Pf(a_{\lambda_i+n-i, \lambda_j+n-j})_{1 \leq i,j \leq n}
\end{align*}
\item Assume that $\ell(\la) < n$ and $n$ is even, so that $x_n=0$. Using \eqref{eq:a_la'} and applying the Schur-Pfaffian identity \eqref{eq:SchurPfaffian}, we obtain
      \begin{align*}
        a_\la(n) &= 2\prod_{k=1}^{n-1} (2k)!\ \Pf\left(\frac{x_i-x_j}{x_i+x_j} \frac{1}{4\,x_i!^2 x_j!^2}\right)_{1 \leq i,j \leq n}\\
&= \prod_{k=1}^{n-1} (2k)!\ \Pf(a_{\lambda_i+n-i, \lambda_j+n-j})_{1 \leq i,j \leq n},
      \end{align*}
with the convention that $a_{0i} = -a_{i0} = 1/(2i!^2)$ and
$a_{00}=0$. The second line is obtained by re-applying the identity
$\det(B)\Pf(A) = \Pf(BAB^t)$ with $B_{ij} = \delta_{ij}/(2x_i!^2)$ for
$1\leq i,j\leq n-1$ and $B_{in} = B_{ni} = \delta_{i,n}$ and $A=\left(\frac{x_i-x_j}{x_i+x_j}\right)_{1
  \leq i,j \leq n}$. The effect is to multiply the coefficients of the
last column and of the last row of $\left(\frac{x_i-x_j}{x_i+x_j}
  \frac{1}{4\,x_i!^2 x_j!^2}\right)_{1 \leq i,j \leq n}$ by $2$. Since
$x_n=0$, we find $(BAB^t)_{in} = - (BAB^t)_{ni} = \delta_{i\neq n}/(2x_i!^2)$.
\item Assume that $\ell(\la)=n$ with $n$ odd. Using \eqref{eq:a_la} and applying the Schur-Pfaffian identity \eqref{eq:SchurPfaffian} for $n$ odd, we obtain
\begin{align*}
a_\la(n) &= \prod_{k=1}^{n-1} (2k)! \det(V_{ij})_{1\leq i,j\leq n+1}\ \Pf\left(\frac{x_i-x_j}{x_i+x_j}\right)_{1 \leq i,j \leq n+1},\\
&= \prod_{k=1}^{n-1} (2k)!\ \Pf(a_{\lambda_i+n-i, \lambda_j+n-j})_{1 \leq i,j \leq n+1},
\end{align*}
with the convention $a_{i,-1} = -a_{-1,i} = 1/(2i!^2)$ and
$a_{-1,-1}=0$. Indeed, the second line is obtained via $\det(B)\Pf(A)
= \Pf(BAB^t)$ with $B=(V_{ij})_{1 \leq i,j \leq n+1}$ and $A
=\left(\frac{x_i-x_j}{x_i+x_j}\right)_{1 \leq i,j \leq n+1}$, so that
the entries on the last column and on the last row are
$(BAB^t)_{i,n+1} = -(BAB^t)_{n+1,i} = \delta_{i\neq n+1}/(2x_i!^2)$.
\item Assume that $\ell(\la)<n$ with $n$ odd, so that $x_n=x_{n+1}=0$. Notice that all elements $a_{ij}$ for $i,j\geq -1$ have been fixed using the three previous cases, except for $a_{0,-1} = -a_{-1,0}$. Using \eqref{eq:a_la'} and applying the Schur-Pfaffian identity \eqref{eq:SchurPfaffian} for $n$ odd, we obtain
\begin{align*}
a_\la(n) &= 2\prod_{k=1}^{n-1} (2k)! \det(V_{ij})_{1\leq i,j\leq n+1}\ \Pf\left(\frac{x_i-x_j}{x_i+x_j}\right)_{1 \leq i,j \leq n+1},\\
&= \prod_{k=1}^{n-1} (2k)! \det(V'_{ij})_{1\leq i,j\leq n+1}\ \Pf\left(\frac{x_i-x_j}{x_i+x_j}\right)_{1 \leq i,j \leq n+1},\\
&= \prod_{k=1}^{n-1} (2k)!\ \Pf(a_{\lambda_i+n-i, \lambda_j+n-j})_{1 \leq i,j \leq n+1},
\end{align*}
with $a_{0,-1} = -a_{-1,0} =1$. Here $V'_{ij} = V_{ij}$ for all $1\leq i,j\leq n+1$ except $V'_{in} = V'_{ni} = 2V_{in} = \delta_{i,n}$.
\end{itemize}
\end{proof}

\subsection{Boson-fermion correspondence and proof of \cref{lemma:virasoroDelayed}}
\label{sec:bosonfermion}
We start by quickly reviewing background material on the half-infinite wedge and the boson-fermion correspondence, \cite{Kac2013Bombay, AlexandrovZabrodin2013}. Denote $V = \bigoplus_{i\in\mathbb{Z}} \C e_i$ the infinite-dimensional complex vector space spanned by linearly independent vectors $e_i$s. For $n\in \mathbb{Z}$, the \emph{half-infinite wedge of charge $n$} is the vector space, denoted $\bigwedge^{\frac{\infty}{2}}_n V$, spanned by the vectors
\begin{equation}\label{eq:defv}
v_{(i_0, i_{-1}, \dotsc)} = e_{i_0} \wedge e_{i_{-1}} \wedge \dotsb,
\end{equation}
where $i_0>i_{-1}> \dotsb $ and there exists $K$ such that for all $k\leq K$, $i_k = k+n$. The half-infinite wedge is then $\bigwedge^{\frac{\infty}{2}} V := \bigoplus_{n\in\mathbb{Z}} \bigwedge^{\frac{\infty}{2}}_n V$. The vacuum of charge $n$ is the vector 
\begin{equation}
|n\rangle := v_{(n, n-1, n-2, \dotsc)}.
\end{equation}

For $i\in\mathbb{Z}$, we denote $\psi_i$ the wedging operator,
\begin{equation}
  \label{eq:wedge}
\psi_i v = e_i\wedge v
\end{equation}
and $\psi^*_i$ the contracting operator defined by its action on basis
\begin{equation}
  \label{eq:contract}
\psi^*_i v_{(i_n, i_{n-1}, \dotsc)} = \sum_{j\leq n} (-1)^{n-j} \delta_{i,i_j} e_{i_n}\wedge \dotsb\wedge e_{i_{j+1}}\wedge e_{i_{j-1}}\wedge \dotsb
\end{equation}
(the sum contains a single non-vanishing term) and extended by
linearity on $\bigwedge^{\frac{\infty}{2}} V$. For positive charges $n$, the vacuum can thus be written $|n\rangle = \psi_{n} \psi_{n-1} \dotsm \psi_1|0\rangle$.

These operators satisfy 
\begin{equation}
\psi_i\psi_j^* + \psi^*_j \psi_i = \delta_{ij}.
\end{equation}
We also introduce $\phi_0$ which acts as the identity on vectors of
even charges and minus the identity on vectors of odd charges. Altogether, the wedging and contracting operators and $\phi_0$ provide a representation of a Clifford algebra on the half-infinite wedge.

For $r\in\mathbb{Z}$, define $\mathcal{J}_r = \sum_{i\in\mathbb{Z}}
:\psi_i \psi^*_{i+r}\colon$, where the colons denote the normal
ordering defined as $:\psi_i \psi^*_{j}\colon = \psi_i \psi^*_{j} -
\delta_{i\leq 0}\delta_{i,j}$. They form a Heisenberg algebra with commutators $[\mathcal{J}_n, \mathcal{J}_m] = n\delta_{n,-m}$.

\begin{proposition} [Boson-fermion correspondence, see~\cite{Kac2013Bombay, AlexandrovZabrodin2013}]
The boson-fermion correspondence is an isomorphism between the half-infinite wedge and $\C[q,q^{-1},p_1, p_2, \dotsc]$. Explicitly, it associates to a vector $|v\rangle$ a polynomial
\begin{equation}
\tau_v = \sum_{n\in\mathbb{Z}} q^n \langle n|\Gamma_+(\pp) |v\rangle,
\end{equation}
where $\Gamma_+(\pp) := e^{\sum_{i>0} \frac{p_i}{i} \mathcal{J}_i}$. 
Moreover, the boson-fermion correspondence turns the above representations of the Heisenberg and Virasoro algebras on the half-infinite wedge into representations over $\C[q,q^{-1},p_1, p_2, \dotsc]$, defined by
\begin{equation} \label{Equivariance}
J_r \tau_v = \tau_{\mathcal{J}_r v}
\end{equation}
and given explicitly by
\begin{equation} \label{CurrentsVirasoro}
J_r = \begin{cases} p_r^* \quad &\text{for $r>0$}\\
p_{-r} \quad &\text{for $r<0$}\\
q\frac{d}{dq} \quad &\text{for $r=0$.} \end{cases}
\end{equation}
\end{proposition}

Note that one can straightforwardly extend this formalism to formal power series. Namely, the boson-fermion correspondence still holds as  an isomorphism between $\C[[q,q^{-1},p_1, p_2,\dotsc]]$ and the space of formal (infinite) linear combinations of $\bigwedge^{\frac{\infty}{2}} V$. Relations \eqref{Equivariance} and \eqref{CurrentsVirasoro} still hold.

The next lemma says that the Virasoro constraints transform nicely under the boson-fermion correspondence. Define the operators on the half-infinite wedge defined by
\begin{equation} \label{FermionicVirasoro}
\mathcal{L}_r = \sum_{i\in\mathbb{Z}} \Bigl(i + \frac{r-1}{2}\Bigr) :\psi_i \psi_{i+r}^*\colon
\end{equation}
They satisfy $[\mathcal{L}_n, \mathcal{L}_m] = (n-m) \mathcal{L}_{n+m} + \frac{1}{12}(n^3-n) \delta_{n,-m}$ and generate a Virasoro algebra with central charge 1.
\begin{lemma}\label{lemma:virasoroBF}
	The boson-fermion correspondence transforms the operators $\mathcal{L}_r$ into operators over $\C[[q,q^{-1}]][p_1, p_2, \dotsc]$ given by
\begin{equation} \label{defLi}
L_r = \begin{cases} q\frac{d}{dq} p_r^* + \frac{1}{2}\sum_{m+n=r} p_m^* p^*_{n} + \sum_{l\geq 1} p_l p_{l+r}^* \qquad &\text{for $r>0$}\\
q\frac{d}{dq} p_r  + \frac{1}{2}\sum_{m+n=-r} p_m p_{n} + \sum_{l\geq 1} p_{l-r} p_{l}^* \qquad &\text{for $r<0$}\\
\frac{1}{2}\bigl(q\frac{d}{dq}\bigr)^2 + \sum_{l\geq 1} p_l p_l^* \qquad &\text{for $r=0$}\end{cases}
\end{equation}
\end{lemma}
This is a classical result in the representation theory of the semi-infinite wedge. As it is of crucial importance to our main result, we provide a proof of this lemma. %

\begin{proof}
The proof below is directly borrowed from unpublished notes by Kac. The operators $\psi_k^*, \psi_{-k+1}$ annihilate the vacuum of charge $0$ for all $k>0$ and are called annihilators, while the operators $\psi_k^*, \psi_{-k+1}$ for $k\leq 0$ are called creators. Annihilators (respectively creators) anti-commute with one another.

The normal order of a product of $\psi_j$ and $\psi_j^*$ is the re-ordering of the operators such that all annihilators are on the right and all creators on the left. More formally, let $\phi^{(1)}_{k_1}, \dotsc, \phi^{(m)}_{k_m}$ be a set of operators such that $\phi^{(q)}_{k_q}\in \{\psi_{k_q}^*, \psi_{-k_q+1}\}$ for all $q=1, \dotsc, m$. Let $\sigma$ be a permutation on $\{1, \dotsc, m\}$ such that $k_{\sigma(1)}\leq \dotsb\leq k_{\sigma(m)}$. The normally ordered product is
\begin{equation*}
:\phi^{(1)}_{k_1} \dotsm \phi^{(m)}_{k_m}\colon = \sgn(\sigma)\
\phi^{(\sigma(1))}_{k_{\sigma(1)}} \dotsm
\phi^{(\sigma(m))}_{k_{\sigma(m)}}.
\end{equation*}
It can be checked that
\begin{equation*}
:\phi^{(1)} \phi^{(2)}\colon = \phi^{(1)} \phi^{(2)} - \langle 0|\phi^{(1)} \phi^{(2)}|0\rangle
\end{equation*}
and \cite{AlexandrovZabrodin2013}
\begin{multline*}
:\phi^{(1)} \phi^{(2)} \phi^{(3)} \phi^{(4)}\colon = \phi^{(1)} \phi^{(2)} \phi^{(3)} \phi^{(4)} \\
- \langle 0| \phi^{(3)} \phi^{(4)} |0\rangle \phi^{(1)} \phi^{(2)} + \langle 0| \phi^{(2)} \phi^{(4)} |0\rangle \phi^{(1)} \phi^{(3)} - \langle 0| \phi^{(2)} \phi^{(3)} |0\rangle \phi^{(1)} \phi^{(4)} \\
- \langle 0| \phi^{(1)} \phi^{(2)} |0\rangle \phi^{(3)} \phi^{(4)} + \langle 0| \phi^{(1)} \phi^{(3)} |0\rangle \phi^{(2)} \phi^{(4)} - \langle 0| \phi^{(1)} \phi^{(4)} |0\rangle \phi^{(2)} \phi^{(3)} \\
+ \langle 0| \phi^{(3)} \phi^{(4)} |0\rangle \langle 0|\phi^{(1)} \phi^{(2)}|0\rangle - \langle 0| \phi^{(1)} \phi^{(3)} |0\rangle \langle 0|\phi^{(2)} \phi^{(4)}|0\rangle \\
+ \langle 0| \phi^{(1)} \phi^{(4)} |0\rangle \langle 0|\phi^{(2)} \phi^{(3)}|0\rangle,
\end{multline*}
where $\phi^{(q)}$ is any of the operators $\psi_j, \psi_j^*$ for
$q=1, \dotsc, 4$. Altogether those relations imply that
\begin{multline} \label{ProductOfNormalOrdered}
:\phi^{(1)} \phi^{(2)}\colon\cdot:\phi^{(3)} \phi^{(4)}\colon =\ :\phi^{(1)} \phi^{(2)} \phi^{(3)} \phi^{(4)}\colon - \langle 0|\phi^{(2)} \phi^{(4)}|0\rangle :\phi^{(1)} \phi^{(3)}\colon \\
+ \langle 0|\phi^{(2)} \phi^{(3)}|0\rangle :\phi^{(1)} \phi^{(4)}\colon - \langle 0|\phi^{(1)} \phi^{(3)}|0\rangle :\phi^{(2)} \phi^{(4)}\colon + \langle 0|\phi^{(1)} \phi^{(4)}|0\rangle :\phi^{(2)} \phi^{(3)}\colon \\
- \langle 0|\phi^{(1)} \phi^{(3)}|0\rangle \langle 0|\phi^{(2)} \phi^{(4)}|0\rangle + \langle 0|\phi^{(1)} \phi^{(4)}|0\rangle \langle 0|\phi^{(2)} \phi^{(3)}|0\rangle.
\end{multline}

We now focus on the Virasoro operators $L_r$ given by \eqref{defLi}, which can be rewritten
\begin{equation*}
L_r = \frac{1}{2} \sum_{k\geq 0} J_{r-k} J_k + \frac{1}{2} \sum_{k\geq 1} J_{-k} J_{r+k}.
\end{equation*}
By the boson-fermion correspondence, they are mapped to 
\begin{equation} \label{QuarticVirasoro}
\begin{aligned}
&\frac{1}{2} \sum_{k\geq 0} \mathcal{J}_{r-k} \mathcal{J}_k + \frac{1}{2} \sum_{k\geq 1} \mathcal{J}_{-k} \mathcal{J}_{r+k}=\\
&\frac{1}{2} \sum_{k\geq 0} \sum_{i,j\in\Z} :\psi_i
\psi^*_{i+r-k}\colon\cdot :\psi_{j}\psi^*_{j+k}\colon + \frac{1}{2} \sum_{k\geq 1} \sum_{i,j\in\Z} :\psi_{i}\psi^*_{i-k}\colon\cdot:\psi_j \psi^*_{j+r+k}\colon
\end{aligned}
\end{equation}
Applying \eqref{ProductOfNormalOrdered}, we find three types of normally ordered terms: quartic, quadratic and constant contributions in $\psi_j, \psi^*_j$. The relevant expectations for this calculation are
\begin{equation*}
\langle 0| \psi_i \psi^*_j|0\rangle = \delta_{i\leq 0} \delta_{i,j},\qquad \langle 0|\psi^*_i \psi_j|0\rangle = \delta_{i>0} \delta_{i,j}
\end{equation*}
while $\langle 0| \psi^*_i \psi^*_j|0\rangle = \langle 0| \psi_i \psi_j|0\rangle = 0$ for all $i,j\in\Z$. The constant terms are immediately found to vanish. The quartic terms are
\begin{align*}
\sum_{k, i, j\in \Z} :\psi_i \psi^*_{i+r-k} \psi_j \psi^*_{j+k}\colon &\underset{k=l+i-j}{=} \sum_{l, i, j\in \Z} :\psi_i \psi^*_{j+r-l} \psi_j \psi^*_{i+l}\colon\\
& = - \sum_{l, i, j\in \Z} :\psi_j \psi^*_{j+r-l} \psi_i \psi^*_{i+l}\colon = 0.
\end{align*}
In the second equality, we use the fact that all operators anti-commute under the normal order. Only quadratic contributions remain and we are left with the quantity
\begin{align*}
&\begin{multlined}
\frac{1}{2} \sum_{\substack{k\geq 0\\ i,j\in\Z}} \Bigl(\langle 0|\psi^*_{i+r-k} \psi_j|0\rangle :\psi_i \psi^*_{j+k}\colon + \langle 0|\psi_i \psi^*_{j+k}|0\rangle :\psi^*_{i+r-k} \psi_j\colon\Bigr) \\
+ \frac{1}{2} \sum_{\substack{k\geq 1\\ i,j\in\Z}} \Bigl(\langle 0|\psi^*_{j-k} \psi_i|0\rangle :\psi_j \psi^*_{i+r+k}\colon + \langle 0|\psi_j \psi^*_{i+r-k}|0\rangle :\psi^*_{j+k} \psi_i\colon\Bigr)
\end{multlined}\\
&= \sum_{i\in\Z} \Bigl(i + \frac{r-1}{2}\Bigr) :\psi_i \psi^*_{i+r}\colon,
\end{align*}
which shows that \eqref{QuarticVirasoro} coincides with \eqref{FermionicVirasoro}.
\end{proof}

A basis of the half-infinite wedge of charge $n$ is given by vectors labeled by partitions. Denote $\lambda = (\lambda_1\geq \lambda_2 \geq \dotsb )$, and as before $\rho_j=\lambda_j-j$ for $j\geq 1$. The state of partition $\lambda$ is a semi-infinite monomial
\begin{equation*}
|n,\lambda\rangle = v_{(\rho_1+n+1, \rho_2+n+1, \dotsc)} = e_{\rho_1+n+1}\wedge e_{\rho_2+n+1}  \wedge\dotsm \wedge e_{\rho_j+n+1}\wedge \dotsm
\end{equation*}
A classical consequence of the boson-fermion correspondence is that those vectors are the image of the Schur symmetric functions, i.e.
\begin{equation} \label{PartitionMapstoSchur}
\langle n|\Gamma_+(\pp) |n, \la\rangle = s_\lambda(\pp).
\end{equation}
We now have all the ingredients to prove \cref{lemma:virasoroDelayed}.

\begin{proof}[Proof of \cref{lemma:virasoroDelayed}]
We first notice that \eqref{eq:VirasoroOnSchur} is equivalent
to
\begin{multline*}
	\biggl(n p_{r}^* + \frac{1}{2}\sum_{m+n=r}p_m^*p_n^*+ \sum_{l \geq 1}
               p_l p^*_{l+r}+\frac{(r-1)}{2}p_r^*-\frac{t}{2}\delta_{r,1}\biggr)s^{(\rho_1, \dots, \rho_k)}(\pp)
	     \\
=-\frac{t}{2}\delta_{r,1} s^{(\rho_1, \dots, \rho_k)}(\pp)
	     +
	     \sum_{i=1}^{k-r} \bigl(n+\rho_i\bigr) s^{(\rho_1, \dotsc, \rho_k)-r \epsilon_i}(\pp),
\end{multline*}
by making the change of variables $\mathbf{p/2} \to \pp$. It can be
further rewritten as
\begin{multline} \label{eq:VirasoroSchurRescaled}
\bigg(L_r + \frac{r-1}{2} p_r^* - \frac{t}{2} \delta_{r,1}\bigg) q^{n} s^{(\rho_1, \dots, \rho_k)}(\pp) =-\frac{t}{2}\delta_{r,1} q^n s^{(\rho_1, \dots, \rho_k)}(\pp)
	    \\ +
	     \sum_{i=1}^{k-r} \bigl(n+\rho_i\bigr) q^n s^{(\rho_1, \dotsc, \rho_k)-r \epsilon_i}(\pp)
\end{multline}
with the Virasoro generators $L_r$ given by \eqref{defLi} (note that
$r \geq 1$). Since by \eqref{PartitionMapstoSchur} Schur symmetric functions at charge $n$ are in
correspondence with the states $|n, \la\rangle$,
we can rewrite the LHS of 
\eqref{eq:VirasoroSchurRescaled} in its fermionic version via
the boson-fermion correspondence at charge $n$ given by \cref{lemma:virasoroBF}:
\begin{equation*}
\mathcal{L}_r^Z |n,\lambda\rangle, \qquad \text{with} \quad \mathcal{L}_r^Z = \mathcal{L}_r +\frac{r-1}{2}\mathcal{J}_r- \frac{t}{2} \delta_{r,1}.	
\end{equation*}
%

%
%
%
%
	By definition of the operators $\mathcal{L}_r, \mathcal{J}_r$ we find,
\begin{align*}
	\mathcal{L}_r^Z |n,\lambda\rangle &=- \frac{t}{2} \delta_{r,1}
  |n,\lambda\rangle 
  + 
  \sum_{j \in \mathbb{Z}}(j+r-1)
	\psi_j\psi^*_{j+r}|n,\lambda\rangle.
\end{align*}
	Now (using again $\rho_i=\lambda_i-i$) the quantity  
	$$\psi_j\psi^*_{j+r}|n,\lambda\rangle = \psi_j\psi_{j+r}^* e_{\rho_1+n+1}\wedge e_{\rho_2+n+1}\wedge \dotsc$$ 
	is nonzero only if there is an index $i$ such that $\rho_i+n+1=j+r$, and in this case it is equal to $$e_{\rho_1+n+1 - r\delta_{i,1}}\wedge e_{\rho_2+n+1-r\delta_{i,2}} \wedge \dotsc=e_{\rho^{(i)}_1+n+1}\wedge e_{\rho^{(i)}_2+n+1}\wedge \dotsc$$
    with $\rho^{(i)}=\rho-r\epsilon_i$.
	Therefore we have (note that necessarily $i\leq k-r$) 
\begin{align*}
	\mathcal{L}_r^Z |n,\lambda\rangle &=- \frac{t}{2} \delta_{r,1}
  |n,\lambda\rangle 
  + 
	\sum_{i=1}^{k-r}(\rho_i+n)e_{\rho^{(i)}_1+n+1}\wedge e_{\rho^{(i)}_2+n+1}\wedge \dotsc
	.
\end{align*}
	We conclude by applying the Boson-Fermion correspondence in the other direction (note that both  $e_{\rho_1+n+1}\wedge e_{\rho_2+n+1}\wedge\dotsc$ and the generalized Schur function $s^{\rho}$  are  antisymmetric in the coordinates of $\rho$, and therefore they are in correspondence even when $\rho$ is not in decreasing order).
\end{proof}

\subsection{The large BKP hierarchy and its ``formal $N$'' solutions}
\label{sec:formalNSolutions}
The following definition of the BKP hierarchy is borrowed from~\cite{KacVandeLeur1998,VandeLeur2001}. 

\begin{definition}[$b_\infty$ and BKP]\label{def:binfty}
The algebra $b_\infty$ (also called quadratic algebra, or spin
algebra) is the Lie algebra of quadratic expressions in the $\psi_i,
\psi^*_j$ and $\phi_0$, where $\psi_i,
\psi^*_j$ are defined by \eqref{eq:wedge}, \eqref{eq:contract} and $\phi_0$ acts as the identity on vectors of
even charges and minus the identity on vectors of odd charges. In
other words
\begin{equation}
b_\infty = \operatorname{Span}\{ \psi_i \psi_j, \psi_i \psi_j^*, \psi^*_i \psi^*_j, \phi_0\psi_i, \phi_0\psi_i^* ; i, j\in\mathbb{Z}\}.
\end{equation}
The exponential of an element of $b_\infty$ is well-defined on $\bigwedge^{\frac{\infty}{2}}V$ and $B_\infty:=\{e^h,h\in b_\infty\}$ is a group.

If a function $\tau \in \C[q,q^{-1}, p_1, p_2,\dotsc]$ is the image via the boson-fermion correspondence of a vector $v$ which lies in the orbit of the vacuum $|0\rangle$ under the action of $B_\infty$, then it is said to be a \emph{polynomial BKP tau function}. Namely, there exists $g\in B_\infty$ such that $v = g|0\rangle$ and
\begin{equation} \label{BKPtau function}
\tau = \tau_v = \sum_{n\in\mathbb{Z}} q^n \tau_{n,v} \qquad \text{with} \quad \tau_{n,v} = \langle n|\Gamma_+(\pp) g |0\rangle.
\end{equation}

		By an abuse of terminology we will also say that the sequence $(\tau_{n,v})_{n\geq 1}$ is a polynomial BKP tau function.
	\end{definition}

We will need to extend the notion of BKP tau functions in two ways. The first one, which is straightfoward, is to allow BKP tau functions  which are not necessarily polynomials. They can be obtained by using formula \eqref{BKPtau function} for 
\begin{equation}
g = \exp \bigg(\sum_{i,j\in\mathbb{Z}} b_{ij} \psi_i \psi_j + a_{ij} \psi_i^* \psi_j^* + c_{ij} \psi_i \psi_j^* + \sum_{k\in\mathbb{Z}} \alpha_k \phi_0 \psi_k^* + \beta_k \phi_0 \psi_k\bigg),
\end{equation}
with arbitrary coefficients $a_{ij}, b_{ij},c_{ij}, \alpha_k, \beta_k
\in \mathbb{C}$, as long as all matrix elements of $g$ are well-defined finite or formal sums.

The other, more subtle, extension consists in allowing tau functions
in which the parameter $n$ is formal rather than an integer, which will be crucial for us. Indeed, the BKP hierarchy requires to evaluate the tau function at integer values of $n$, but our main function is \emph{not} defined for $u=1/(2n)$ with integer $n$, since its coefficients have poles. We thus need a certain regularization, which will be performed by considering truncations of the Schur function expansion. This will in turn allow for the definition of formal BKP solutions, i.e. functions which depend on a \emph{formal} parameter denoted $N$, instead of $n$ integer.
This requires some care since our main function can be made a tau function only after normalizing it by a certain function of $n$ (defined a priori for integer, non-formal, $n$).
In the context of matrix integrals, some of these subtleties are hidden in the fact that the underlying specialization of the variables $\pp$ to powersums of finitely many eigenvalues realize in some sense this truncation explicitly, see~\eqref{eq:matrixTrunc} in the next section. The next definition enables us to work at the formal level and without specializations.

\begin{definition}[Formal-$N$ normalized BKP tau function]\label{def:formalBKP}
	Let $\tau(N) \in \mathbb{Q}(N)[[p_1, p_2,\dotsc]]$ be a formal power series with coefficients which are rational functions of the formal parameter $N$.
	For each $L\in \mathbb{N}$, consider the following truncation
        of the Schur-function expansion of $\tau(N)$, in which we only
        keep Schur functions corresponding to partitions with at most $L$ parts:
	$$
	\mathrm{Trunc}(\tau,L)(N) := 
	 [s_\cdot ^{\ell(\cdot\leq L)}] \tau(N) = \sum_{\mu: \ell(\mu)\leq L}  ( [s_\mu] \tau(N)) s_\mu. $$
	 Assume that there exists a sequence $(\alpha_K)_{K\geq 1}$ tending to infinity, such that for each integer value $K\in \mathbb{N}$ the coefficients of the function $\mathrm{Trunc}(\tau,\alpha_K)(N)$ have no pole at $N=i$ for any integer $i\geq K$. In particular, the function $\mathrm{Trunc}(\tau,\alpha_K)(K)$ is well defined.

	Moreover, assume that there exists a sequence $(\beta_N)$ such
        that the sequence 
	$$\beta_K \mathrm{Trunc}(\tau,\alpha_K)(K),  K\in\mathbb{Z} $$
	is a tau function of the BKP hierarchy and such that for all $N\geq 0$ and every odd positive integer $k$ one has
\begin{align}\label{eq:2factorial}
	\frac{\beta_{N-1}\beta_{N+k-1}}{\beta_{N}\beta_{N+k-2}} =R_k(N)
\end{align}
for some rational function $R_k(N)\in \mathbb{Q}(N)$ and
for all $N,k\geq 0$ one has
\begin{align}\label{eq:3factorial}
	\frac{\beta_{N-2}\beta_{N+k}}{\beta_{N}\beta_{N+k-2}} =S_k(N)
\end{align}
for some rational function $S_k(N)\in \mathbb{Q}(N)$.
Then we say that $\beta_N\tau(N)$ is a \emph{formal-$N$ BKP tau function}.
\end{definition}
In the last sentence, we slightly abuse notation by writing
``$\beta_N\tau(N)$'' since $\beta_N$ is a function on integers while
$\tau(N)$ is a formal object which is not well-defined on
integers. Being completely rigorous, it is the pair $((\beta_n)_{n\geq 1}, \tau(N))$ which encodes the data of the formal-$N$ tau function, but the notation we use should be clear enough  in what follows. We have

\begin{proposition}[BKP hierarchy~\cite{KacVandeLeur1998}, stated here for formal $N$]\label{prop:BKP}
	Let $\beta_N\tau(N)$ be a formal-$N$ BKP tau function. Then for $k\in\mathbb{N}, k\geq 1$
the following bilinear identity holds in  $\mathbb{C}(N)[\pp,\qq][[t]]$,
	\begin{multline} \label{BKP}
		\frac{1}{2}\bigl((-1)^k-1\bigr) \frac{\beta_{N-1} \beta_{N+k-1}}{\beta_N \beta_{N+k-2}} U(\qq) \tau(N-1)\cdot \tau(N+k-1) 
\\
		+ 
		\frac{\beta_{N-2} \beta_{N+k}}{\beta_{N} \beta_{N+k-2}} \sum_j h_j(2\qq) h_{j-k-1}(-\mathbf{\check{D}}) U(\qq) \tau(N-2) \cdot \tau(N+k) \\
		+ \sum_j h_j(-2\qq) h_{j+k-1}(\mathbf{\check{D}}) U(\qq) \tau(N) \cdot \tau(N+k-2) = 0
\end{multline}
	where $\qq = (q_1, q_2, \dots)$ is a vector of formal indeterminates. Here we use the following notation:
	\begin{itemize}[itemsep=0pt, topsep=0pt,parsep=0pt, leftmargin=12pt]
\item $h_j$ is the complete homogeneous symmetric function of degree $j$,
\item $\mathbf{\check{D}} = (k D_k)_{k\geq 1}$ where $D_r$ is the Hirota derivative with respect to $p_r$ defined as the following bilinear mapping
\begin{equation}
\bigl(f(p_r),g(p_r)\bigr) \ \mapsto\ D_r f\cdot g (p_r)= \frac{\partial}{\partial s_r} f(p_r + s_r) g(p_r - s_r)_{|s_r=0}
\end{equation}
\item $U(\qq) = e^{\sum_{r\geq 1} q_r D_r}$.
\end{itemize}
\end{proposition}
\begin{remark}
	Here it is crucial to note that for fixed $k$, the quantities
        $\bigl((-1)^k-1\bigr)\frac{\beta_{N-1} \beta_{N+k-1}}{\beta_N \beta_{N+k-2}}$ and
        $\frac{\beta_{N-2} \beta_{N+k}}{\beta_{N} \beta_{N+k-2}}$
        appearing in~\eqref{BKP} are elements of $\mathbb{Q}(N)$,
        because of the hypothesis~\eqref{eq:2factorial} and~\eqref{eq:3factorial}. Therefore~\eqref{BKP} is \emph{truly} a \emph{formal} identity satisfied by the (formal, unnormalized, non-truncated) generating function $\tau(N)$.
	A sufficient condition for~\eqref{eq:2factorial} and~\eqref{eq:3factorial} to hold is that $\beta_N\beta_{N+2}/\beta_{N+1}^2$ is a rational function in $N$. This will be the case for the main function studied in this paper, but not for the cases mentioned in Appendix~\ref{sec:otherModels}.
\end{remark}

Proposition~\ref{prop:BKP} provides an infinite list of partial
differential equations on the function $\tau(N)$ and its shifts by
looking at coefficients of the Taylor expansion of~\eqref{BKP} in the
variables $\qq$. For example, the ``first'' equation of this
hierarchy, often called \emph{the BKP-equation} is obtained by taking $k=2$ and extracting the coefficient of
$q_3$. It is given by
\begin{align}\label{eq:explicitBKP}
\sminus{}F_{3,1}(N) \splus{} F_{2,2}(N) \splus{} \tfrac{1}{2} F_{1,1}(N)^2 \splus{} \tfrac{1}{12} F_{1,1,1,1}(N)=
	\footnotesize{S_2(N)} \frac{\tau(N\!\sminus{}\!2) \tau(N\!\splus{}\!2)}{\tau(N)^2},
\end{align}
where we use the notation $F(N) = \log \tau(N)$ and $f_{i} = \frac{\partial f}{\partial p_i}$. Here we have used that $\beta(N+2)\beta(N-2)/\beta(N)^2 =S_2(N)$. 

\begin{proof}[Proof of Proposition~\ref{prop:BKP}]
	By definition of a formal-$N$ BKP tau function we know \sloppy
        that
        $\beta(N)\mathrm{Trunc}(\tau,\alpha_N)(N)$ is a BKP-tau
        function. Therefore, the fundamental result of Kac and Van de
        Leur~\cite{KacVandeLeur1998} asserts that~\eqref{BKP} holds true
        provided we replace $\tau(N)$ with
        $\mathrm{Trunc}(\tau,\alpha_N)(N)$.
	Now, fix two partitions  $\lambda$ and $\mu$, and consider the
        coefficient of $[p_\lambda q_\mu]$ in the equation thus
        obtained. This coefficient is a certain polynomial identity
        relating a finite number of coefficients of the functions
        $\mathrm{Trunc}(\tau,\alpha_m)(m)$ for $m\in \{N, N-1,N-2,
        N+k-2, N+k-1, N+k\}$. We recall that the coefficients of the
        function $\tau(\cdot)$ are rational functions in $N$. In
        particular, taking $N$ large enough (i.e. $N>N_0(\lambda, \mu,
        k)$) we can see that this polynomial identity is an
         identity between rational functions in $N$, which holds true
        for infinitely many values of $N$. Therefore this polynomial identity in fact holds by replacing each coefficient involved by the corresponding rational function. This is precisely what~\eqref{BKP} is saying. 
\end{proof}

\begin{corollary}
	Let 
	$F(N)=\sum_\lambda b_\lambda(N) s_\lambda t^{|\lambda|}$
	be a formal power series in $\mathbb{C}(N)[\pp][[t]]$ and assume that 
	there exists an infinite matrix $B=(B_{i,j})_{i,j\in \mathbb{Z}}$ such that for every integer $N$ and partition $\lambda$ such that $\ell(\lambda) \leq N$, one has
	\begin{align}\label{eq:pfaffianCoeff}
		\beta_N b_\lambda (N)=
		\begin{cases} \Pf(B_{\lambda_i+N-i,\lambda_j+N-j})_{1 \leq i,j \leq N}\ \text{ for $N$ even,}\\
 \Pf(B_{\lambda_i+N-i,\lambda_j+N-j})_{1 \leq i,j \leq N+1}\ \text{ for $N$ odd,} 
		\end{cases}
	\end{align}
		where $(\beta_N)_{N\geq 1}$ is a sequence satisfying~\eqref{eq:2factorial}.
	Then $\beta_NF(N)$ is a formal-$N$ BKP tau function. 
\end{corollary}
\begin{proof}
	The proof is based on a variant of the Wick theorem, which enables us to express the Pfaffians in~\eqref{eq:pfaffianCoeff} in the infinite-wedge formalism. 
	Define the elements 
	$$
 g_1:= \sum_{i \geq 0} B_{i,-1} \psi_i \phi_0
		\ \ , \ \
	g_2:=\sum_{j>i \geq 0} B_{i,j} \psi_i \psi_j  , 
	$$
	which belong to $b_\infty$, as infinite formal linear combinations.

	First, if $N$ is even, we claim that 
	\begin{align}\label{eq:evenPfaff}
		\Pf (B_{\lambda_i-i+N,\lambda_j-j+N})_{1 \leq i,j\leq N}=
\frac{1}{(N/2)!} \langle N,\lambda | g_2^{N/2} | 0, \emptyset \rangle.
	\end{align}
	To see this, just  expand the $(N/2)$-th power $g_2^{N/2}$ and observe that the terms contributing to the scalar product expression are in one-to-one correspondence with ordered pairings of the set of indices $(\lambda_i-i+N)_{1\leq i,j \leq N}$, where each pair $(i,j)$ comes with a weight $B_{i,j}$. The factorial comes from~\eqref{eq:DefPf} while the absence of power of $2$ compared to that equation comes from the fact that the definition of $g_2$ only contains terms with $i<j$. Note moreover that the signs coming from anticommutation of Fermionic operators match the sign in~\eqref{eq:DefPf}.

	The situation is similar if $N$ is odd, with now
\begin{align}\label{eq:oddPfaff}
	\Pf (B_{\lambda_i+N-i,\lambda_j+N-j})_{1 \leq i,j \leq N+1} = \frac{1}{((N-1)/2)!} \langle N,\lambda | g_1 g_2^{(N-1)/2} | 0, \emptyset \rangle.
\end{align}
	To see this, it is important to observe that the index $\lambda_i+N-i$ appearing in the Pfaffian can be equal to $-1$,  if $i=N+1$, $\lambda_i=\lambda_{N+1}=0$. Therefore a term $\psi_i \phi_0$ coming from the expansion of $g_1$ can be interpreted as matching the index $i$ with the index $-1$ in the definition~\eqref{eq:DefPf} of the Pfaffian. Note also that the factorial is $((N-1)/2)!$ and not $((N+1)/2)!$ since by considering the product $g_1g_2^{(N-1)/2}$ we force the pair containing the index $-1$ to appear in first position.

	Now observe that, because $g_i$ is an operator of charge $i$ for $i=1,2$, both \eqref{eq:evenPfaff} for $N$ even, and \eqref{eq:oddPfaff} for $N$ odd, are equal to $\langle N,\lambda | (1+g_1) e^{g_2} | 0, \emptyset \rangle$, since in each case only one order of the Taylor expansion of the exponential contributes.

In the notation of Definition~\ref{def:formalBKP} it follows that the truncated function, with $\alpha_N=N$, can be expressed as
	\begin{align*}
		\beta_N \mathrm{Trunc}(\tau,N)(N) = 
 \sum_{\lambda\atop \ell(\lambda)\leq N}
	\beta_Nb_\lambda(N) s_\lambda t^{|\lambda|}
		=  \langle N | \Gamma_+(\pp) (1+g_1) e^{g_2} | 0 \rangle.	
	\end{align*}
	Finally, observe (the situation is similar to~\cite{VandeLeur2001}) that $(1+g_1)\in B_\infty$ since $g_1^2=0$ as seen by expanding the square and using anticommutation relations, so that $1+g_1 = e^{g_1}$.
	Hence $\beta_N \mathrm{Trunc}(\tau,N)(N)$ is a BKP tau function, which concludes the proof.
\end{proof}

Combining the results above we obtain one of our main results that
says that the generating function of $b$-monotone Hurwitz numbers is a
formal-$N$ BKP tau function (up to an explicit prefactor $\beta_N$) for $b=1$ and $2N=u^{-1}$.

\begin{theorem}[BKP hierarchy for non-oriented monotone Hurwitz numbers]\label{thm:mainBKP}
	The function
	$$\beta_N\tau^Z_{b=1}(t; \mathbf{2p}, u)\Big|_{u=(2 N)^{-1}}
	=\beta_N\sum_{n \geq 0} t^n\sum_{\lambda \vdash n} 
        \frac{Z_\lambda(\mathbf{2p})}{j_\lambda}\prod_{\square
	\in \lambda}\frac{1}{ 2 N +c_b(\square)},
	$$
	with $\beta_N=\tfrac{1}{\prod_{k=1}^{N-1}(2k)!}$,
	is a formal-$N$ BKP tau function. It satisfies the hierarchy
        of equations~\eqref{BKP}, and in particular the BKP
        equation~\eqref{eq:explicitBKP} with
        $S_2(N)=\frac{1}{(2N+2)_4(2N)_4}$, where $(N)_k :=
        N(N-1)\cdots(N-k+1)$ denotes the falling factorial.
\end{theorem}
\begin{proof}
	This is a direct consequence of the formalism developped in this section, together with the explicit expansion of Theorem~\ref{thm:schurExpansion}
$$
	\tau_{b=1}^{Z}\Bigl(t;\mathbf{ 2 p},\frac{1}{2N}\Bigr) = \sum_{k\geq 0} t^k \sum_{\lambda\vdash k}
	a_\lambda(N) 
	s_\lambda(\mathbf{p})$$
	and the  Pfaffian structure of $a_\lambda(N)$ given by Theorem~\ref{theo:Pfaffian},
\begin{equation*}
\operatorname{Trunc}(\tau_{b=1}^Z, n)\Bigl(t;\mathbf{2p}, \frac{1}{2n}\Bigr) = \sum_{k\geq0} t^k \sum_{\substack{\la\vdash k\\ \ell(\la)\leq n}} a_\la(n) s_\la(\pp).
\end{equation*}
Note in particular that $\beta_N$ satisfies~\eqref{eq:2factorial} and~\eqref{eq:3factorial} with
	$R_k(N)=\frac{1}{(2N+2k-4)_{2k-2}}$ and $S_k(N) =
        \frac{1}{(2N+2k-2)_{2k}(2N+2k-4)_{2k}}$.
\end{proof}

\section{Orthogonal BGW integral and its solution}
\label{sec:BGW}

\subsection{The BGW integral over $O(n)$}

The \emph{Br\'ezin--Gross--Witten integral} (BGW integral for short)
is a function of a complex, $n\times n$ matrix $X$ and its complex
conjugate $X^\dagger$, defined as an integral over the unitary group $U(n)$
\begin{equation}
\operatorname{BGW}_{U(n)}(X, t) := \int_{U(n)} dU\ \exp \sqrt{t}\,\operatorname{tr} (UX + X^\dagger U^\dagger),
\end{equation}
where $dU$ denotes the normalized Haar measure over $U(n)$.	The variable $t$ can be completely re-absorbed in $X$ and $X^\dagger$, but we will keep it to make the link with our function $\tau_b^Z(t;\pp,u)$. As emphasized in \cite{MironovMorozovSemenoff1996}, this integral has several interesting expansions. We will focus on the ``character'' expansion (in the terminology of \cite{MironovMorozovSemenoff1996}), meaning that we consider $\operatorname{BGW}_{U(n)}(X, t)$ as an element of $\QQ[X][[t]]$, where $\QQ[X]$ denotes the ring of polynomials in the matrix elements of $X$ and its complex conjugate. Then, for $X\in GL_n(\C)$, $\operatorname{BGW}_{U(n)}(X, t)$ admits an expansion which coincides with $\tau_b^Z(t;\pp,u)$ at $b=0$, $u=1/n$ and $p_i = p_i(XX^\dagger) := \operatorname{tr}((XX^\dagger)^i)$ being the power-sums associated with the $n$ eigenvalues of $XX^\dagger$. We will describe below this equality, and how this evaluation of $\tau_b^Z(t;\pp,u)$ makes sense.

For an $n\times n$ matrix $Y$ with eigenvalues $y_1, \dotsc, y_n$, denote 
\begin{equation*}
\mathbf{p}(Y) = (p_i(Y):=\operatorname{tr}(Y^i))_{i\geq1}, \qquad \text{and} \quad s_\lambda(Y) := s_\lambda(\mathbf{p}(Y))
\end{equation*}
the Schur polynomial evaluated on the variables $y_1, \dotsc, y_n$
(recall that we conventionally parametrize symmetric functions by the
power-sums of the underlying variables). We start by expanding
$\operatorname{BGW}_{U(n)}(X, t)$ on Schur polynomials, following
ideas from \cite{ZinnJustin2002} used in an analogous situation of the
HCIZ integral. Using the Cauchy identity \cite[(2.1), (2.2)]{ZinnJustin2002}:
\[ \exp(t\operatorname{tr}X) = \sum_{k \geq 0}\sum_{\lambda \vdash k}t^k\frac{s_\lambda(X)}{\hook_\lambda}\]
and commuting sums with integral we have
\begin{equation}
\operatorname{BGW}_{U(n)}(X, t) = \sum_{k,l\geq 0} t^{\frac{k+l}{2}} \sum_{\lambda\vdash k, \mu\vdash l} \frac{1}{\hook_\lambda\hook_\mu} \int_{U(n)} dU\ s_\lambda(UX) s_\mu(X^\dagger U^\dagger).
\end{equation}
For $X\in GL_n(\C)$, the following orthogonality relation is classical
and follows from the orthogonality relation of the irreducible
characters of $U(n)$, which are given by Schur polynomials (see~\cite{ZinnJustin2002})%
\begin{equation} \label{UnitarySchurIntegral}
\int_{U(n)} dU\ s_\lambda(UX) s_\mu(X^\dagger U^\dagger) = \begin{cases} \delta_{\lambda, \mu} \frac{s_\lambda(XX^\dagger)}{s_\lambda(1^n)}\quad &\text{if $\ell(\lambda)\leq n$,}\\
0 \quad &\text{else.} \end{cases}
\end{equation}
Notice in particular that the denominator $s_\lambda(1^n)$ is always well-defined for $n\in\N$. This gives 
\begin{equation}
\operatorname{BGW}_{U(n)}(X, t) = \sum_{k\geq 0} t^k \sum_{\substack{\lambda\vdash k\\ \ell(\lambda)\leq n}} \frac{1}{\hook_\lambda^2\, s_\lambda(1^n)} s_\lambda(XX^\dagger).
\end{equation}
This expansion is exactly $\tau_{b=0}^Z(t;\pp,u)$, truncated as
$\operatorname{Trunc}(\tau_{b=0}^Z,n)$ (\cref{def:formalBKP}) and
evaluated at $u=1/n$, $p_i = p_i(XX^\dagger)$, which yields
\begin{equation} \label{UnitaryBGW_SchurExpansion}
\operatorname{BGW}_{U(n)}(X, t) = \operatorname{Trunc}(\tau_{b=0}^Z,n)(t;\pp(XX^\dagger), 1/n).
\end{equation}

As already emphasized, our generating series $\tau_b^Z$ cannot be evaluated on $u^{-1}\in\N$ because the coefficients $[s_\lambda]\tau_b^Z$ are rational in $u^{-1}$ with poles on integers. However the truncation of $\tau_{b=0}^Z$ following \cref{def:formalBKP}, reduces the expansion on Schur functions to those partitions for which $\ell(\lambda)\leq n$. Then the coefficients $[s_\lambda] \operatorname{Trunc}(\tau_{b=0}^Z,n)$ have poles only on integers $u^{-1} = i<n$ and the evaluation at $u^{-1}=n$ makes sense.

There is a natural truncation to $\ell(\lambda)\leq n$ in the BGW
model, which is implemented via \eqref{UnitarySchurIntegral} and
ultimately due to the fact that there is a \emph{finite} number of
independent variables $x_1, \dotsc, x_n$. Indeed, by \emph{first}
substituting the variables $p_1, p_2, \dotsc$ with $p_1(XX^\dagger),
p_2(XX^\dagger), \dotsc$ in $\tau_{b=0}^Z$, all the Schur polynomials with $\ell(\la)>n$ vanish. This gives
\begin{equation}\label{eq:matrixTrunc}
\operatorname{Trunc}(\tau_{b=0}^Z,n)(t; \pp(XX^\dagger), 1/n) = \tau_{b=0}^Z(t;\pp(XX^\dagger), u)_{\vert u=1/n}
\end{equation}

There is no direct way to write a $b$-deformation (or $\beta$-deformation, in the context of matrix integrals, related by $1+b=2/\beta$) of the BGW integral, since there is no group which deforms $U(n)$ over which to integrate. Nevertheless, there is an obvious candidate for an orthogonal BGW integral,
\begin{equation}
\operatorname{BGW}_{O(n)}(X, t) := \int_{O(n)} dO\ \exp \sqrt{t}\operatorname{tr} (XO),
\end{equation}
where $dO$ is the normalized Haar measure over $O(n)$ and $X$ can be an arbitrary $n\times n$ matrix. It is connected to our main function as follows.

\begin{proposition}[Matrix integral expression of $\tau^Z_{b=1}$] \label{prop:BGWON}
Let $X\in GL_n(\RR)$ and denote $X^t$ its transpose. Then,
\begin{equation*}
\operatorname{BGW}_{O(n)}(X, t) = \tau^Z_{b=1}(t; \pp(XX^t), u)_{\vert
  u^{-1}=n}
\end{equation*}
\end{proposition}

\begin{proof} \cite[Section VII.3, Ex. 2]{Macdonald1995} The proof parallels the unitary case. One first expands the orthogonal BGW integral on zonal polynomials. This is done with the Cauchy identity, and commuting sum with integral. The classical integral \eqref{UnitarySchurIntegral} is replaced with
\begin{equation}
\int_{O(n)} dO\ s_\lambda(XO) = \begin{cases} \frac{Z_\mu(XX^t)}{Z_\mu(1^n)} \quad &\text{if $\lambda =2\mu$ and $\ell(\lambda)\leq n$,}\\
0 \quad &\text{else,}\end{cases}
\end{equation}
for $X\in GL_n(\RR)$. Here $Z_\mu(XX^t)$ is the zonal polynomial evaluated on the power-sums $p_i = p_i(XX^t)$. One gets
\begin{equation}
\operatorname{BGW}_{O(n)}(X, t) = \sum_{k\geq 0} \sqrt{t}^{k} \sum_{\lambda \vdash k} \int_{O(n)}\frac{s_\lambda(XO)}{\hook_{\lambda}}=\sum_{k\geq 0} t^{k} \sum_{\substack{\lambda \vdash k\\ \ell(\lambda)\leq n}} \frac{1}{\hook_{2\lambda}\,Z_\lambda(1^n)} Z_\lambda(XX^t).
\end{equation}

We now show that the right hand side of \cref{prop:BGWON} is the same as the above zonal expansion. Since $Z_\la(XX^t)=0$ when $\ell(\la)> n$ \cite{Stanley1989}, the evaluation of $\tau_{b=1}^Z$ on the power-sums $p_i(XX^t)$ reduces to
\begin{equation*}
\tau_{b=1}^Z(t;\mathbf{p(XX^t)}, u) = \sum_{k\geq 0} t^k \sum_{\substack{\lambda \vdash k\\ \ell(\lambda)\leq n}} \frac{1}{\hook_{2\lambda}\,\prod_{\square \in\la} (u^{-1} + c_1(\square))} Z_\lambda(XX^t)
\end{equation*}
One concludes with $Z_\la(1^n) = \prod_{\square \in\la} (n + c_1(\square))$ \cite{Stanley1989}, which does not vanish as long as $\ell(\la)\leq n$.
\end{proof}

\subsection{Pfaffian solutions}

The BGW integral in the unitary case admits also a representation as a ratio of determinants. If $XX^\dagger \in GL_n(\C)$ has eigenvalues $x_1, \dotsc, x_n$, denote $\Delta(XX^\dagger) = \det (x_i^{n-j}) = \prod_{1\leq i<j\leq n} (x_i-x_j)$ the corresponding Vandermonde determinant. Then \cite{MironovMorozovSemenoff1996}
\begin{equation} \label{BGWasDet}
\operatorname{BGW}_{U(n)}(X, t) = \left(\prod_{k=1}^{n-1} k!\right)\, \frac{\det \Bigl((t x_i)^{\frac{n-j}{2}} I_{n-j}\bigl(2 \sqrt{t x_i}\bigr)\Bigr)_{1\leq i,j\leq n}}{t^{\frac{n(n-1)}{2}} \Delta(XX^\dagger)}
\end{equation}
where
\[I_j(z) = \bigg(\frac{z}{2}\bigg)^j \sum_{l \geq 0}
  \frac{(z^2/4)^l}{l!(l+j)!}\]
is the $j$-th modified Bessel function of the first kind.
The derivation of \eqref{BGWasDet} in
\cite{MironovMorozovSemenoff1996} makes use of the HCIZ integral and
HCIZ formula\footnote{It is also possible to derive this expression from the matrix model of \cite{ZinnJustin2002}, also using the HCIZ formula.}
\begin{equation*}
\operatorname{HCIZ}_{U(n)}(X,Y) := \int_{U(n)} dU\ \exp \operatorname{tr} (XUYU^{-1}) = \frac{\det (e^{x_i y_j})}{\Delta (X) \Delta(Y)}.
\end{equation*}
In order to find an analogous formula for $\operatorname{BGW}_{O(n)}$, one
could use an orthogonal version of the HCIZ integral, where the
integral is over $O(n)$, $X$ is an arbitrary matrix, and $Y\in
O(n)$. However the Harish-Chandra formula, which works for any compact
Lie group $G$, is only valid for $X,Y$ in the Lie algebra of $G$, and
in the case $G=O(n)$ it applies only to skew-symmetric
matrices $X,Y$, therefore it cannot be used to derive a formula for for $\operatorname{BGW}_{O(n)}$. A compact formula for the HCIZ integral over the orthogonal
group with generic $X, Y$ (say real, symmetric) is unknown and posed
as an important open problem~\cite{BergereEynard2009}.

Nevertheless, it is natural to expect that an analogous formula to
\eqref{BGWasDet} exists in the orthogonal case and involves Pfaffians.
Our second main result gives such a formula in the special case when $X$
is arbitrary but the eigenvalues of $XX^t$ all have even
multiplicities. Equivalently, they are $x_1, x_1, x_2, x_2, \dotsc,x_{n}, x_{n}$ (we allow $x_i=x_j$ for $i \neq j$).

\begin{theorem}[The orthogonal BGW integral as a Pfaffian]\label{thm:pfaffianBGW}
Assume that $XX^t \in GL_{2n}(\RR)$ has eigenvalues $x_1, x_1, x_2, x_2, \dotsc, x_n, x_n$. Denote $\hat{X}_n = \operatorname{diag}(\sqrt{x_1}, \dotsc, \sqrt{x_n})$. Then for $n$ even
\begin{equation}
\operatorname{BGW}_{O(2n)}(X, t) = \left(\prod_{k=1}^{n-1} (2k)!\right)\, \frac{\operatorname{Pf}\bigl(M(t;x_i,x_j)\bigr)_{1\leq i,j\leq n}}{t^{\frac{n(n-1)}{2}} \Delta (\hat{X}_n\hat{X}^t_n)},
\end{equation}
where
\begin{multline}
M(t;x,y) = \frac{1}{8} \int_0^{t} \frac{dt'}{\sqrt{t'}} \Bigl(\sqrt{x} I_1(2\sqrt{t'x})\big(1+I_0(2\sqrt{t'y})\big) - \sqrt{y} I_1(2\sqrt{t'y})\big(1+I_0(2\sqrt{t'x})\big)\Bigr).
\end{multline}
and for $n$ odd
\begin{equation}
\operatorname{BGW}_{O(2n)}(X, t) = \left(\prod_{k=1}^{n-1} (2k)!\right)\, \frac{\operatorname{Pf}\bigl(M_{ij}(t;x_1,\dotsc, x_n)\bigr)_{1\leq i,j\leq n+1}}{t^{\frac{n(n-1)}{2}} \Delta (\hat{X}_n\hat{X}^t_n)},
\end{equation}
where
\begin{equation}
M_{ij}(t;x_1,\dotsc, x_n) = \begin{cases} M(t;x_i,x_j) \qquad &\text{for $i,j=1, \dotsc, n$},\\
 \frac{1}{2} I_0(2\sqrt{tx_i}) &\text{for $i=1, \dotsc, n$ and $j=n+1$},\\
-\frac{1}{2} I_0(2\sqrt{tx_j}) &\text{for $j=1, \dotsc, n$ and $i=n+1$}.
\end{cases}
\end{equation}
\end{theorem}
\begin{proof}
The power-sums of $XX^t$ and $\hat{X}_n$ are related by
\begin{equation*}
\pp(XX^t)\mathbf{/2} = \pp(\hat{X}_n\hat{X}^t_n).
\end{equation*}
Proposition~\ref{prop:BGWON} expresses
$\operatorname{BGW}_{O(2n)}(X,t)$ as $\tau^Z_{b=1}(t; \pp(XX^t), u)_{\vert u^{-1}=2n}$,
which we can further express in the scaled Schur basis by \cref{thm:schurExpansion}
\begin{equation*}
\begin{aligned}
\tau_{b=1}^Z(t;\pp(XX^t),u)_{\vert u^{-1}=2n} &= \sum_{m\geq 0} t^m \sum_{\substack{\lambda\vdash m\\ \ell(\lambda)\leq n}} \frac{1}{\hook_{\lambda}^2\,o_\lambda(1^{u^{-1}})}_{|u^{-1}=2n} s_\lambda(\pp(\hat{X}_n\hat{X}^t_n)).
\end{aligned} 
\end{equation*}
Note that the sum over partitions $\lambda$ runs a priori over all
possible partitions, but $s_\lambda(\pp(\hat{X}_n\hat{X}^t_n))$
vanishes for $\ell(\lambda)>n$.

Now we use \cref{theo:Pfaffian} to express the coefficient
$\frac{1}{\hook_{\lambda}^2\,o_\lambda(1^{u^{-1}})}_{|u^{-1}=2n} =
a_\lambda(n)$ with $n\in\N$ as a Pfaffian, and the bialternant formula
to express $s_\lambda(\pp(\hat{X}_n\hat{X}^t_n))$ as a ratio of
determinants. For $n$ even we obtain
\begin{equation*}
\operatorname{BGW}_{O(2n)}(X, t) = \left(\prod_{k=1}^{n-1} (2k)!\right)
\sum_{\substack{\lambda:\\ \ell(\lambda)\leq n}} \operatorname{Pf}(a_{\lambda_i+n-i,\lambda_j+n-j})_{1\leq i,j\leq n} \frac{\det \bigl((tx_i)^{\lambda_j+n-j}\bigr)_{1\leq i,j\leq n}}{t^{\frac{n(n-1)}{2}}\Delta(\hat{X}_n\hat{X}^t_n)},
\end{equation*}
and from the minor summation formula for Pfaffians given by
\cref{prop:CBForPfaffians} with $N=\infty$ it can be rewritten as
\begin{equation*}
\operatorname{BGW}_{O(2n)}(X, t) = \left(\prod_{k=1}^{n-1} (2k)!\right) \frac{\operatorname{Pf}\bigl(M(t;x_i,x_j)\bigr)_{1\leq i,j\leq n}}{t^{\frac{n(n-1)}{2}}\,\Delta(\hat{X}_n\hat{X}^t_n)},
\end{equation*}
with
\begin{equation} \label{EvenPfaffianBGW}
\begin{aligned}
M(t;x_i,x_j) &= \sum_{k,l\geq 0} t^{k+l}\ x_i^k a_{kl} x_j^l = \sum_{k\geq 1} t^k \frac{x_i^k - x_j^k}{2 k!^2} + \sum_{k,l\geq 1} t^{k+l}\frac{k-l}{4(k+l)} \frac{x_i^k}{k!^2} \frac{x_j^l}{l!^2}\\
&= \int_0^{t} \frac{dt'}{t'} \sum_{k\geq 1} k \frac{(t'x_i)^k -
  (t'x_j)^k}{2 k!^2} + \frac{1}{4} \int_0^{t} \frac{dt'}{t'}
\sum_{k,l\geq 1} (k-l) \frac{(t'x_i)^k}{k!^2} \frac{(t'x_j)^l}{l!^2}\\
&= \frac{1}{4} \int_0^{t} \frac{dt'}{t'} \Bigg(\sum_{k\geq 1}
k\frac{(t'x_i)^k}{k!^2} \bigg(2+\sum_{l \geq 1}\frac{(t'x_j)^l}{l!^2}\bigg)-\sum_{l\geq 1}
l\frac{(t'x_j)^k}{l!^2} \bigg(2+\sum_{k \geq
  1}\frac{(t'x_i)^l}{k!^2}\bigg)\Bigg)\\
&= \frac{1}{8} \int_0^{t} \frac{dt'}{\sqrt{t'}} \Bigl(\sqrt{x_i} I_1(2\sqrt{t'x_i})\big(1+I_0(2\sqrt{t'x_j})\big) - \sqrt{x_j} I_1(2\sqrt{t'x_j})\big(1+I_0(2\sqrt{t'x_i})\big)\Bigr).
\end{aligned}
\end{equation}

For $n$ odd,
\begin{equation*}
\operatorname{BGW}_{O(2n)}(X, t) = \left(\prod_{k=1}^{n-1} (2k)!\right)
\sum_{\substack{\lambda:\\ \ell(\lambda)\leq n}} \operatorname{Pf}(a_{\lambda_i+n-i,\lambda_j+n-j})_{1\leq i,j\leq n+1} \frac{\det \bigl((tx_i)^{\lambda_j+n-j}\bigr)_{1\leq i,j\leq n}}{t^{\frac{n(n-1)}{2}}\Delta(\hat{X}_n\hat{X}^t_n)},
\end{equation*}
In order to apply the minor summation formula for Pfaffians, we have
to promote $\det (x_i^{\lambda_j+n-j})_{1\leq i,j\leq n}$ to an
$(n+1)\times (n+1)$ determinant, by adding a row and a column of
zeroes to $(x_i^{\lambda_j+n-j})_{1\leq i,j\leq n}$ except for a 1 on
the diagonal. Therefore
\begin{multline*}
\operatorname{BGW}_{O(2n)}(X, t) = \frac{(-1)^{\frac{n(n-1)}{2}}}{t^{\frac{n(n-1)}{2}}\,\Delta(\hat{X}_n\hat{X}^t_n)} \left(\prod_{k=1}^{n-1} (2k)!\right)\\ 
\sum_{m\geq 0} \sum_{\substack{\lambda\vdash m\\ \ell(\lambda)\leq n}} \operatorname{Pf}(a_{\lambda_i+n-i,\lambda_j+n-j})_{1\leq i,j\leq n+1} \det (X_{i,\lambda_j+n-j})_{1\leq i,j\leq n+1}
\end{multline*}
with
\begin{equation*}
X_{ij} = \begin{cases} (tx_i)^{j} \qquad &\text{for $i=1, \dotsc, n$ and $j\geq 0$}\\
0 & \text{for $i=n+1$ and $j\geq 0$, and for $j=-1$ and $i=1, \dotsc, n$}\\
1 & \text{for $i=n+1$, $j=-1$}.
\end{cases}
\end{equation*}
Now, applying the minor-summation formula for Pfaffians given by
\cref{prop:CBForPfaffians} with $N=\infty$ we end up with
\[ \operatorname{BGW}_{O(2n)}(X, t) = \left(\prod_{k=1}^{n-1}
    (2k)!\right)\, \frac{\operatorname{Pf}\bigl(M_{ij}(t;x_1,\dotsc,
    x_n)\bigr)_{1\leq i,j\leq n+1}}{t^{\frac{n(n-1)}{2}} \Delta
    (\hat{X}_n\hat{X}^t_n)},\]
where
\begin{equation*}
M_{ij}(t;x_1,\dotsc, x_n) = \sum_{k,l\geq -1} X_{ik} a_{kl} X_{lj} =
\sum_{k\geq 0} a_{k,-1}\bigl(\delta_{i,n+1} (tx_j)^k - \delta_{j,n+1}
(tx_i)^k\bigr) + \sum_{k,l\geq 0} (tx_i)^k a_{kl} (tx_j)^l
\end{equation*}
and $x_{n+1}:=0$. The first sum
\[ \sum_{k\geq 0} a_{k,-1}\bigl(\delta_{i,n+1} (tx_j)^k - \delta_{j,n+1}
(tx_i)^k\bigr)= \delta_{i,n+1}\frac{1}{2} I_0(2\sqrt{tx_j})
-\delta_{j,n+1}\frac{1}{2} I_0(2\sqrt{tx_i}),\]
while the second was already computed in \eqref{EvenPfaffianBGW}.
\end{proof}

\appendix
\section{Related models and final comments}
\label{sec:otherModels}

In this appendix we mention analogies between monotone Hurwitz numbers and two other well studied models, namely general maps and bipartite maps. The results mentioned here are not original, but it is not always easy to find a canonical reference.

We will be interested in the two following $b$-deformed tau functions, which are special cases of the main function of~\cite{ChapuyDolega2020}
\begin{equation} \label{MapsTauFunctions}
\begin{aligned}
\tau_b^{(1)}(t;\pp,u) &= \sum_{n \geq 0} t^n\sum_{\lambda \vdash n} 
        \frac{J_\lambda(\pp)\theta_{2^n}(\lambda)}{j_\lambda}\prod_{\square
                              \in \lambda}(u-c_b(\square)), \\
                \tau_b^{(2)}(t;\pp,u,v) &= \sum_{n \geq 0} t^n\sum_{\lambda \vdash n} 
        \frac{J_\lambda(\pp)}{j_\lambda}\prod_{\square
                                      \in \lambda}(u-c_b(\square))(v-c_b(\square)),
\end{aligned}
\end{equation}
where $\theta_\mu(\lambda) = [p_\mu]J^{(1+b)}_\lambda(\pp)$.

By~\cite[Theorem~3.14]{ChapuyDolega2020}, $\tau^{(1)}$ (respectively, $\tau^{(2)}$) can be interpreted as the generating function of all maps (respectively, all bipartite maps) on non-oriented surfaces, with a weight $t$ per edge, a weight $u$ per vertex (respectively $u$ per white vertex and $v$ per black vertex), and a weight $p_i$ per face of degree $i$ (respectively, half-degree $i$). In both cases, each map $\mathcal{M}$ is equipped with a monomial $b$-weight of the form $b^{\nu(\mathcal{M})}/(1+b)^{cc(M)}$ where $\nu(\mathcal{M})$ is a non-negative integer, which vanishes if and only if the underlying surface is orientable. 

It is classical that $\tau^{(1)}_b$ for $b=0$ and $b=1$ coincide
respectively with the GUE and GOE partition functions. The general
$b$-case corresponds to the interpolation given by $\beta$-ensembles,
with $(1+b)\beta=2$. In the context of matrix integrals, it is also
known (see e.g.~\cite{AdlervanMoerbeke2001}) that their partition functions satisfy Virasoro constraints, and therefore a quantum curve equation that can be obtained by resummation. Similar statements can be made about $\tau^{(2)}$ and will not come as a surprise to experts.

Let $V(x,\pp) = \sum_{i\geq 1} \frac{p_i}{i} x^i$, and denote $\Delta_N(x) = \det (x_i^{j-1})_{1\leq i,j\leq N}$ the Vandermonde determinant for a set of $N$ variables. Let
\begin{equation*}
\begin{aligned}
Z^{(1)}_b(t;\pp,N) &= \int_{\mathbb{R}^N} \Bigl(\prod_{i=1}^N dx_i e^{-\frac{x_i^2}{2(1+b)t^2} + V(x_i;\mathbf{\frac{p}{1+b}})}\Bigr) |\Delta_N(x)|^{\frac{2}{1+b}}\\
Z^{(2)}_b(t;\pp,N,M) &= \int_{\mathbb{R}_+^{N}} \Bigl(\prod_{i=1}^{N} dx_i e^{-\frac{x_i}{(1+b)t} + V(x_i;\mathbf{\frac{p}{1+b}})} \lambda_i^{\frac{M-N+1}{1+b} -1}\Bigr) |\Delta_{N}(x)|^{\frac{2}{1+b}}
\end{aligned}
\end{equation*}
for $M\geq N$. Then, for $N\in\mathbb{N}$,
\begin{equation} \label{MatrixIntegrals}
\tau^{(1)}_b(t;\pp,N) = \frac{Z^{(1)}_b(t;\pp,N)}{Z^{(1)}_b(t;0,N)} \qquad \tau^{(2)}_b(t;\pp,N,M) = \frac{Z^{(2)}_b(t;\pp,N,M)}{Z^{(2)}_b(t;0,N,M)}
\end{equation}

Equation~\eqref{MatrixIntegrals} can be proved in at least two ways.
\begin{itemize}[itemsep=0pt, topsep=0pt,parsep=0pt, leftmargin=12pt]
\item By proving the same Virasoro constraints (or the same evolution equation) for both sides, see below.
\item By expanding the matrix integrals on Jack symmetric functions\footnote{One expands
\begin{equation*}
e^{\sum_{i=1}^N V(x_i;\mathbf{\frac{p}{1+b}})} = \sum_\la J^{(1+b)}_\la(x_1, \dotsc, x_N)\, \frac{J^{(1+b)}_\la(\pp)}{j_\la^{(b)}},
\end{equation*}
using the Cauchy identity for Jack symmetric functions, then commutes the sum over partitions with the integrals and evaluate the integrals using \cite{Kadell1997, Kaneko1993}. This is essentially the method proposed in \cite{MironovMorozovShakirov2010} for the $b$-deformed BGW integral.}.
\end{itemize}

\medskip

The Virasoro constraints for the matrix integrals can be obtained by standard techniques \cite{AdlervanMoerbeke2001}. Furthermore, the Virasoro constraints for $\tau^{(1)}_b$ and $\tau^{(2)}_b$, as defined in \eqref{MapsTauFunctions} can be directly obtained from Lemma~\ref{lemma:evolutionImpliesVirasoro}, giving the following.
\begin{proposition}
	For $m=1,2$ we have $L_i^{(m)} \tau^{(m)} =0$ for $i\geq -\delta_{m,1}$, where $(L_i^{(m)})_{i\geq -\delta_{m,1}}$ generate the Virasoro algebra and are given by
\begin{align*}
         L^{(1)}_i &= \begin{multlined}[t]\frac{p_{i+2}^*}{t^2} - \bigg((1+b)\sum_{\substack{m,n\geq1\\m+n=i}}p_m^*p_n^*+ \sum_{n \geq 1} p_n p^*_{n+i} + (b(i+1)+2u)p_i^* 
+ \tfrac{\delta_{i,-1}up_1+u(u+b)\delta_{i,0}}{1+b}\bigg),\end{multlined}\\
        L^{(2)}_i &= \frac{p_{i+1}^*}{t}
-\bigg((1+b)\sum_{\substack{m,n\geq1\\m+n=i}}p_m^*p_n^*+ \sum_{n \geq 1} p_n p^*_{n+i} + (bi+u+v)p_i^*+\frac{uv\delta_{i,0}}{1+b}\bigg).
\end{align*}
\end{proposition}
\begin{proof}
	Theorem~6.1 in \cite{ChapuyDolega2020} provides an evolution
        equation for each model, which can be written explicitly. It
        is then a direct check that
        Lemma~\ref{lemma:evolutionImpliesVirasoro} applies again,
        respectively with $s=t$ and $k=2,1$, respectively.
\end{proof}

\medskip
From the existence of matrix models, it is standard that the functions $\tau^{(1)}$ and $\tau^{(2)}$, at $b=0$, respectively $b=1$, are related to KP tau functions, respectively BKP tau functions. At $b=0$ the KP tau functions are
\begin{align*}
\sum_{\lambda} \det\biggl( \int_{\RR} dx\ e^{-\frac{x^2}{2t^2}} x^{\lambda_i+2N-i-j}\biggr)_{1\leq i,j\leq N} s_\lambda(\pp) &= \frac{Z^{(1)}_{b=0}(t;0,N)}{N!} \tau^{(1)}_{b=0}(t;\pp,N),\\
\sum_{\lambda} \det\biggl( \int_{\RR_+} dx\ e^{-\frac{x}{t}} x^{\lambda_i+N+M-i-j}\biggr)_{1\leq i,j\leq N} s_\lambda(\pp) &= \frac{Z^{(2)}_{b=0}(t;0,N,M)}{N!} \tau^{(2)}_{b=0}(t;\pp,N,M).
\end{align*}
Those equations can be obtained starting from \eqref{MatrixIntegrals} and using the Cauchy identity to expand $e^{\sum_{i=1}^N V(x_i;\pp)}$ on Schur functions. Then one commutes the sum over partitions with the integrals and use the bialternant formula for the Schur polynomials, yielding
\begin{equation*}
Z^{(1)}_{b=0}(t;\pp,N) = \sum_\lambda s_\lambda(\pp) \int_{\RR^N} \Bigl(\prod_{i=1}^N dx_i e^{-\frac{x_i^2}{2t^2}}\Bigr) \det x_i^{N-j}\ \det x_i^{\lambda_j+N-j}.
\end{equation*}
We conclude by using Andreev's formula which gives
\begin{equation*}
Z^{(1)}_{b=0}(t;\pp,N) = N! \sum_\lambda s_\lambda(\pp) \det\biggl( \int_{\RR} dx\ e^{-\frac{x^2}{2t^2}} x^{\lambda_i+2N-i-j}\biggr)_{1\leq i,j\leq N}.
\end{equation*}
The same calculation can be performed with $Z^{(2)}_{b=0}(t;\pp,N,M)$. Notice that both $Z^{(1)}_{b=0}(t;0,N)$ and $Z^{(2)}_{b=0}(t;0,N,M)$ are Selberg integrals which can be evaluated explicitly \cite{Mehta2004}.
 
At $b=1$, one has for $N$ even
\begin{multline} \label{GeneralMapsBKP}
\sum_{\lambda} \Pf\biggl( \int_{\RR^2} dx dy\ e^{-\frac{x^2+y^2}{4t^2}} x^{\lambda_i+N-i} \sgn(x-y) y^{\lambda_j+N-j}\biggr)_{1\leq i,j\leq N} s_\lambda(\pp) \\
= \frac{Z^{(1)}_{b=1}(t;0,N)}{N!} \tau^{(1)}_{b=1}(t;\mathbf{2p},N)
\end{multline}
and 
\begin{multline*}
\sum_{\lambda} \Pf\biggl(\int_{\RR_+^2} dx dy\ e^{-\frac{x+y}{2t}} x^{\lambda_i+N+\frac{\delta-1}{2}-i} \sgn(x-y) y^{\lambda_j+N+\frac{\delta-1}{2}-j}\biggr)_{1\leq i,j\leq N} s_\lambda(\pp) \\
= \frac{Z^{(2)}_{b=1}(t;0,N,N+\delta)}{N!} \tau^{(2)}_{b=1}(t;\mathbf{2p},N,N+\delta).
\end{multline*}
Similar formulas can be obtained for $N$ odd, proving that they are BKP tau functions. Again $Z^{(1)}_{b=1}(t;0,N)$ and $Z^{(2)}_{b=1}(t;0,N,N+\delta)$ are well-known Selberg integrals. Let us prove \eqref{GeneralMapsBKP} using techniques similar to \cite{VandeLeur2001}. Following the same steps as for $b=0$, one arrives at
\begin{align*}
Z^{(1)}_{b=1}(t;\mathbf{2p},N) &= \sum_\lambda s_\lambda(\pp) \int_{\RR^N} \Bigl(\prod_{i=1}^N dx_i e^{-\frac{x_i^2}{4t^2}}\Bigr) \sgn\bigl(\Delta_N(x)\bigr) \det x_i^{\lambda_j+N-j}\\
&= \sum_\lambda s_\lambda(\pp) \int_{\RR^N} \Bigl(\prod_{i=1}^N dx_i e^{-\frac{x_i^2}{4t^2}}\Bigr) \Pf \bigl(\sgn(x_i-x_j)\bigr)\ \det x_i^{\lambda_j+N-j}
\end{align*}
One concludes by using de Bruijn's formula, which extends Andreev's to Pfaffians.

Instead of using matrix integrals, the above equations can be derived
using the techniques of \cref{sec:schurExpansion} and
\cref{sec:BKP}. There is however a subtle and nevertheless crucial
difference with the case of monotone Hurwitz numbers we consider
here. It is indeed not necessary for general and bipartite maps to
consider truncated expansions: the functions can be evaluated at any
integer value of $N$, and the relation with matrix models is in a
sense much stronger for these models than for the monotone case. In
the case $b=0$, the monotone case is known to be more complicated, and
such subtleties have generated quite a lot of interest
\cite{Novak2020}. Finally note that, although it is not necessary to
consider $N$ as a formal variable in those two models, it is still
possible to do it, and the BKP hierarchy holds over $\mathbb{Q}(N)$ in
the exact sense we made precise in this paper -- note that
$\beta^{(1)}_N = Z^{(1)}_{b=1}(t;0,N)/N!$ and $\beta^{(2)}_N =
Z^{(2)}_{b=1}(t;0,N,N+\delta)/N!$ satisfy~\eqref{eq:2factorial}
and~\eqref{eq:3factorial} with
\begin{align}
  R^{(1)}_k&=\left(\frac{N+k-1}{2}\right)_{\frac{k-1}{2}}\frac{N}{N+k-1}(\sqrt{2}t)^{k-1},\\
  S^{(1)}_k&=(N+k)_k\frac{N(N-1)}{(N+k)(N+k-1)}t^k
\end{align}
in the first
case and
\begin{align}
  R^{(2)}_k&=\left(\frac{N+k-1}{2}\right)_{\frac{k-1}{2}}\left(\frac{N+k+\delta-3}{2}\right)_{\frac{k-1}{2}}\frac{N}{N+k-1}(2t)^{k-1},\\
  S^{(2)}_k&=(N+k)_k (N+\delta+k-2)_k\frac{N(N-1)}{(N+k)(N+k-1)}t^{2k}
\end{align}
in the second
case. These formulas are straightforward from the explicit form of
the Selberg integrals $Z^{(1)}_{b=1}(t;0,N)$ and
$Z^{(2)}_{b=1}(t;0,N,N+\delta)$ computed in~\cite{Mehta2004}.

\bibliographystyle{amsalpha}%
\bibliography{biblio2015}%
\end{document}